\documentclass[a4paper,11pt,]{article}
\usepackage{amsmath,amssymb,amsthm}
\usepackage{mathtools,dsfont}
\usepackage{paralist,enumitem,bm,placeins,url}
\usepackage{color,graphicx,subfig} 
\usepackage[margin=30mm,top=40mm]{geometry}
\usepackage{cite}

\usepackage{tabularx}
\usepackage{booktabs}
\usepackage{indentfirst}
\usepackage{array}
\newcolumntype{C}[1]{>{\centering\arraybackslash}p{#1}}
\usepackage{tocbasic}
\DeclareTOCStyleEntry[
beforeskip=.2em plus 1pt,% default is 1em plus 1pt
]{tocline}{section}
\usepackage{hyperref} 
\hypersetup{%
	colorlinks=true,
	linkcolor=blue,%
	citecolor=red,%
	urlcolor=blue
}
\usepackage[normalem]{ulem}
\usepackage{mathrsfs}
\usepackage{pgfplots}
\usepackage{caption}
\pgfplotsset{compat=1.7}
\usepackage[title]{appendix}

\newtheorem{thm}{Theorem}[section]
\newtheorem{lem}[thm]{Lemma}
\newtheorem{rem}[thm]{Remark}

\newcommand{\vol}{\operatorname{vol}}
\newcommand*{\vv}[1]{\vec{\mkern0mu#1}}
\newcommand{\norm}[1]{\Vert#1\Vert}
\newcommand{\bR}{{\mathbb R}}
\newcommand{\mT}{\mathscr{T}}
\newcommand{\mQ}{\mathcal{Q}}
\newcommand{\bkap}{{\overline{\varkappa}}}
\newcommand{\Gt}{{\Gamma(t)}}
\newcommand{\Gm}{{\Gamma^m}}
\newcommand{\bV}{\mathbb{V}}
\newcommand{\dH}{{\rm d}\mathcal{H}}

\newcommand{\id}{{\rm id}}
\newcommand{\dd}[1]{\frac{\rm d}{{\rm d}#1}}
\newcommand{\ddt}{\dd{t}}
\newcommand{\nn}{\nonumber}
\newcommand{\ttau}{\Delta t}
\newcommand{\ipd}[1]{\bigl(#1\bigr)}
\newcommand{\nabs}{\nabla_{\!s}}

\newcommand{\mat}[1]{\uuline{#1}\rule{0pt}{0pt}}

\numberwithin{equation}{section}

{{\upshape\bfseries AMS subject classifications. }\ignorespaces}{}

\begin{document}
\title{%
Geometric structure-preserving parametric finite element approximations for the constrained Helfrich flow  
}

\author{Xiaoxiao Liu\footnotemark[1] 
    \and Quan Zhao\footnotemark[2]}

\renewcommand{\thefootnote}{\fnsymbol{footnote}}

\footnotetext[1]{School of Mathematical Sciences, University of Science and Technology of China, 230026 Hefei, China  \\ \tt(\href{mailto:xxl0226@mail.ustc.edu.cn}{xxl0226@mail.ustc.edu.cn})}

\footnotetext[2]{School of Mathematical Sciences, University of Science and Technology of China, 230026 Hefei, China \\
\tt(\href{mailto:quanzhao@ustc.edu.cn}{quanzhao@ustc.edu.cn}) }

\date{}
\maketitle

\begin{abstract}
We propose a structure-preserving parametric finite element method for the constrained Helfrich flow of closed curves and surfaces. The proposed method is based on a two-stage velocity-splitting strategy. In the first stage, the
normal velocity is computed from a curvature evolution equation, with the volume and surface area constraints imposed softly in terms of the normal velocity. This step approximates the gradient-flow structure of the Helfrich flow, and its fully discrete parametric finite element approximation leads to a linear system and yields an unconditional energy dissipation estimate at the fully discrete level. In the second stage, the surface mesh is updated by combining the computed normal velocity with a BGN-type tangential velocity. A time-weighted interface normal and an area-correction multiplier are also used to enforce exact preservation of the enclosed volume and surface area. This correction step leads to a nonlinear system, which can be efficiently solved by Newton iteration. The resulting method simultaneously achieves energy decay, exact geometric conservation, and good mesh quality. Numerical experiments for two-dimensional curves and three-dimensional surfaces, including nonsmooth initial data and nonzero spontaneous curvature, are presented to demonstrate the accuracy, robustness, and structure-preserving properties of the proposed method.
\end{abstract}


\setlength{\parindent}{2em}

\renewcommand{\thefootnote}{\arabic{footnote}}

\setcounter{equation}{0}

%=============================================Introduction===============================================
\section{Introduction} \label{sec:intro}

The evolution of surfaces driven by curvature energies arises in a broad range of applications, including vesicle relaxation, geometric processing, materials science, and applied mathematics. One of the most prominent examples is the bending energy described by the Helfrich functional~\cite{Canham1970minimum,Helfrich73elastic},
\begin{equation}\label{eq:Helfrich_general}
E_H(\Gamma)=\int_{\Gamma}\left[\frac{k_c}{2} (\varkappa - \bkap)^2+k_G K\right]\, \dH^{d-1} 
\end{equation}
where $\Gamma$ is a closed hypersurface in $\mathbb{R}^d$ ($d=2,3$), $\varkappa$ and $K$ denote the mean and Gaussian curvatures, $k_c>0$ is the bending rigidity, $k_G$ is the Gaussian bending rigidity, $\bkap \in \mathbb{R}$ stands for the spontaneous curvature, and $\dH^{d-1}$ represents the integration with respect to the $(d-1)$-dimensional Hausdorff measure in $\bR^d$. For closed surfaces with fixed topology, the Gauss--Bonnet theorem implies that $\int_{\Gamma} K \, \dH^{d-1}$ is a constant. Hence, the Gaussian curvature term does not influence the variational dynamics and may be omitted.

For simplicity, we set $k_c = 1$ and therefore consider the total energy of the following form: 
\begin{equation}\label{eq:Helfrich_reduced}
E_{\bkap}(\Gamma)=\frac{1}{2}\int_{\Gamma}(\varkappa - \bkap)^2\, \dH^{d-1}.
\end{equation}
In most scenarios, it is also necessary to incorporate constraints on both the surface area and the enclosed volume for a more physically realistic description. Such constraints are fundamental in phenomena such as vesicle budding and cell shape transformations. This leads to the following minimization problem: 
\begin{equation}\label{eq:mini}
\min_{\Gamma}E_{\bkap}(\Gamma)\quad \text{s.t.}\quad \left\{\begin{array}{ll}
\text{(i)} & |\Gamma|:=\int_{\Gamma} 1 \dH^{d-1}=A_{0}; \\[0.5em]
\text{(ii)} & \vol(\Gamma):=\frac{1}{d} \int_{\Gamma}\vec\id \cdot \vec\nu \dH^{d-1}=V_{0};
\end{array}\right.
\end{equation}
where $|\Gamma|$ and $\vol(\Gamma)$ represent the surface area and enclosed volume of $\Gamma$, respectively, $A_0$ and $V_0$ are prescribed constants, $\vec\id$ denotes the identity map in $\bR^d$, and $\vec\nu$ is the outer unit normal to $\Gamma$. The minimization problem \eqref{eq:mini} gives rise to a geometric evolution equation via its $L^2$-gradient flow, known as the Helfrich flow (or Willmore--Helfrich flow); see \eqref{eq:3d_velocity}. The flow is a highly nonlinear fourth-order geometric evolution equation, whose coupling with global surface area and volume constraints poses significant challenges for accurate and stable numerical approximation.  

In recent decades, significant effort has been devoted to numerical approximations of geometric flows driven by curvature. A prominent line of research focuses on structure-preserving parametric finite element methods (PFEM), which are designed to ensure that key geometric properties of the continuous flow are preserved at the discrete level; see, e.g., \cite{BMN05,Dziuk08,pwfade,Barrett20,BZ21SPFEM,BGNZ22volume,BLani23} and \cite{Kemmochi25structure,Duan25,GJSZ25,BaoL25,GNZ26,GGLT26}. Recent advances have also highlighted the crucial role of tangential motion in maintaining mesh quality in parametric approximations. A variety of strategies have been proposed, including the BGN framework (Barrett, Garcke, and N\"urnberg)~\cite{BGN08parametric}, the MDR (minimal deformation rate) approach~\cite{Hu22evolving}, the DeTurck trick~\cite{DeTurck17}, as well as other related techniques~\cite{Remacle10,Duan24new,PAN26}. In this work, we propose a novel parametric finite element method for the Helfrich flow. In particular, we aim to preserve the intrinsic geometric properties of the Helfrich flow while simultaneously achieving high mesh quality within a unified variational framework. 

We next briefly review numerical approximations of the pure geometric Helfrich flow, with emphasis on parametric finite element methods. For related
developments on the Willmore flow, we refer the reader to the recent works~\cite{GNZ25willmore,GNZ26} and the references therein. BGN-type parametric finite element methods for the Helfrich flow, which exploit tangential degrees of freedom to improve mesh quality, were developed in \cite{BGN08willmore,pwfade}. The schemes in \cite{BGN08willmore} exhibit good mesh distribution properties, while the approach in \cite{pwfade} provides stable semi-discrete approximations and includes volume- and surface-area-preserving variants. However, for the fully discrete variants in these works, a simultaneous guarantee of energy dissipation, exact volume conservation, and exact surface area conservation is generally not available. The axisymmetric case was further investigated in \cite{pwfopen}, where the reduced geometric setting allows the constraint Lagrange multipliers to be treated more directly, leading to schemes with exact volume and surface area preservation. Other related numerical approaches include \cite{Elliott10,BONITO2010,chen2015}, as well as structure-preserving methods for planar curve flows~\cite{Kemmochi25structure}. In particular,
\cite{Kemmochi25structure} proposed a structure-preserving approximation for the constrained Helfrich flow of planar closed curves based on the discrete gradient method, preserving both energy dissipation and the geometric constraints.

Our work is inspired by the recent work \cite{GNZ26}, which proposed a fully energy-stable finite element method for the Willmore flow based on a two-stage velocity-splitting strategy. In the first stage, the gradient flow is approximated in terms of the normal velocity with the help of the curvature evolution equation. In the second stage, the computed normal velocity is used to update the evolving interface with a suitably chosen tangential velocity for better mesh quality. In the present work, we generalize this idea to the constrained Helfrich flow. We also adopt a stagewise strategy for the treatment of the geometric constraints. In the first stage, the volume and surface area constraints are imposed approximately through the normal velocity and the associated Lagrange multipliers, which we refer to soft constraints.  This treatment is compatible with the stability estimate of the underlying gradient-flow structure. The remaining geometric drifts are then corrected in the second stage. More precisely, the volume constraint can be enforced exactly with the help of the discrete time-weighted interface normals introduced in \cite{BZ21SPFEM}, while surface area preservation can again be enforced through a hard constraint with an additional Lagrange multiplier. The proposed method achieves, within a unified framework, unconditional energy decay, machine-precision constraint preservation, and excellent mesh quality. 

The rest of the paper is organized as follows. In Section~\ref{sec:mathform}, we derive the geometric PDE system for the constrained Helfrich flow and present its weak formulation together with the energy law and geometric preservation properties. In Section~\ref{sec:pfem}, we introduce the parametric finite element approximations, including the approximation of the gradient-flow structure, the mesh update, and the practical variants. There, we rigorously prove the energy stability and structure-preserving properties of the proposed method. In Section~\ref{sec:num}, we report a series of experiments for both planar curves and three-dimensional surfaces, with particular attention to energy decay, constraint preservation, mesh quality, spontaneous curvature effects, and nonsmooth initial data. Finally, conclusions and possible extensions are given in Section~\ref{sec:con}.

\section{Mathematical formulations}\label{sec:mathform}

Let $(\Gt)_{t\in[0,T]}$ be an evolving %$C^2$-
hypersurface in $\bR^d$ with its parameterization given by
\begin{equation}\label{eq:para}
\vec x(\cdot, t): \Upsilon\times[0,T]\mapsto\bR^d,\qquad d\in\{2,3\},
\end{equation}
where $\Upsilon\subset\bR^d$ is a fixed oriented 
reference manifold without boundary. The material velocity of $\Gt$ under this parameterization is defined as
\begin{equation}
\label{eq:velocity}
\vec{\mathscr{V}}(\vec x(\vec\rho,t), t) = \partial_t\vec x(\vec\rho,t)\qquad\forall(\vec\rho,t)\in\Upsilon\times [0,T].
\end{equation}
We also introduce the normal velocity of the surface as
\begin{equation}
\mathscr{V}(\vec x, t) = \vec{\mathscr{V}}(\vec x, t)\cdot\vec\nu(\vec x, t)\qquad\mbox{on}\quad\Gamma(t),
\end{equation}
where $\vec\nu$ is the unit normal to $\Gt$.

\subsection{The Helfrich flow and its new geometric PDE system}
With a slight abuse of notation, we now write the total energy in \eqref{eq:Helfrich_reduced} as  
\begin{equation}\label{eq:Henergy}
E_{\bkap}(\Gamma(t), \varkappa(t)) = \frac{1}{2}\int_{\Gamma(t)}(\varkappa-\bkap)^2\dH^{d-1},
\end{equation}
where for simplicity we denote $\varkappa(t) = \varkappa(\cdot, t)$. Here the curvature is defined by
\begin{equation*}
	\varkappa = -\nabs\cdot\vec\nu\qquad\mbox{on}\quad\Gt,
\end{equation*}
where $\nabs$ is the surface gradient operator implicitly defined on $\Gamma(t)$. Our sign convention is such that $\varkappa=-(d-1)$ for the unit sphere with outer normal. Using the transport theorem, one obtains
\begin{align} \label{eq:dtE}
\ddt E_{\bkap}(\Gamma(t),\varkappa(t)) &= \int_{\Gt}\left[\Delta_s\varkappa + (\varkappa-\bkap)|\nabs\vec\nu|^2 - \frac{1}{2}(\varkappa-\bkap)^2\varkappa\right]\mathscr{V}\dH^{d-1},
\end{align} 
where $\nabs\vec\nu$ is the Weingarten map, and $|\mat{A}|^2={\rm tr}(\mat{A}\,\mat{A}^T)$ is the Frobenius norm for any matrix $\mat{A}\in\bR^{d\times d}$. In the case $d=2$, $|\nabs\vec\nu|^2$ reduces to $\varkappa^2$. We also have the curvature identity 
\begin{equation}
	\varkappa\,\vec\nu = \Delta_s\vec\id\qquad\mbox{on}\quad \Gt,\label{eq:curidentity}
\end{equation}
where $\Delta_s=\nabs\cdot\nabs$ is the Laplace-Beltrami operator.

Taking the $L^2$-gradient flow of the energy $E_{\bkap}$ together with the volume and area constraints then yields the desired normal velocity of the evolving surface: 
\begin{equation}\label{eq:3d_velocity}
\mathscr V=-\Delta_s\varkappa-(\varkappa-\bar\varkappa)|\nabla_s\vec\nu|^2+\frac12(\varkappa-\bar\varkappa)^2\varkappa+\lambda(t)+\mu(t)\varkappa\quad\text{on }\Gamma(t),
\end{equation}
where $\lambda(t)$ and $\mu(t)$ are the Lagrange multipliers enforcing exact volume and area preservation:
\begin{subequations}\label{eq:conservation_laws}
\begin{align}
\frac{\mathrm{d}}{\mathrm{d} t}\vol(\Gamma(t)) & = \int_{\Gamma(t)} \mathscr{V} \dH^{d-1} = 0, \quad t \geq 0, \\
\frac{\mathrm{d}}{\mathrm{d} t}|\Gamma(t)| &= -\int_{\Gamma(t)} \mathscr{V} \varkappa \dH^{d-1} = 0, \quad t \geq 0.
\end{align}
\end{subequations}

We next follow the idea in \cite{GNZ26} and consider the material derivative of the curvature 
\begin{equation}
\partial_t^\circ \varkappa = \Delta_s\mathscr{V} + \mathscr{V}\,|\nabla_s\vec\nu|^2 + \mathscr{\vv V}\cdot\nabla_s\varkappa.
\end{equation}
Here, the first two terms on the right-hand side describe the change in curvature induced by the normal velocity, while the additional convective term $\mathscr{\vv V}\cdot\nabla_s\varkappa$ accounts for the contribution from the tangential velocity. This motivates the introduction of a new geometric PDE system for the Helfrich flow on $\Gamma(t)$.
\begin{subequations}\label{eq:newpde}
\begin{align}\label{eq:gf1}
\mathscr{V} &= -\Delta_s \varkappa -(\varkappa-\bar{\varkappa})|\nabla_s \vec\nu|^2+\frac12(\varkappa-\bar{\varkappa})^2\varkappa+\lambda + \mu\kappa, \\
\label{eq:gf2}
\partial_t^\circ \varkappa &= \Delta_s \mathscr{V}+ \mathscr{V}|\nabla_s\vec\nu|^2+ \vec{\mathscr{V}}\cdot\nabla_s \varkappa,\\
\label{eq:gf3}
0 &=\int_{\Gamma(t)}\mathscr{V}\dH^{d-1},\\
\label{eq:gf4}
0 &=\int_{\Gamma(t)}\mathscr{V}\kappa\dH^{d-1},\\
\label{eq:gf5}
\mathscr{\vv V}\cdot\vec\nu  &= \mathscr{V},\\
\kappa\,\vec\nu &= \Delta_s\vec\id.
\label{eq:gf6}
\end{align}
\end{subequations}
Here, \eqref{eq:gf1}--\eqref{eq:gf4} are used to approximate the gradient-flow structure, while \eqref{eq:gf5}--\eqref{eq:gf6} incorporate the normal velocity together with the desired BGN tangential velocity. We refer to $\varkappa$ as the evolution curvature, computed from the time
evolution equation, and to $\kappa$ as the geometric curvature, computed directly from the geometric surface. These two quantities coincide at the continuous level but differ after discretization. This separation allows the design of appropriate tangential velocities while ensuring discrete energy stability.

\subsection{Weak formulation}\label{sec:weak}

To formulate a weak form of the new PDE system \eqref{eq:newpde}, we denote by $\ipd{\cdot, \cdot}_{\Gamma}$ the $L^2$-inner product on $\Gamma(t)$. The following antisymmetric treatment of the convective term $\vec{\mathscr{V}}\cdot\nabla_s\varkappa$ in the curvature transport equation plays a key role in the stability estimate; its proof can be found in \cite[Lemma 3.1]{GNZ26}.
\begin{lem}
Let $\Gamma(t)$ be a closed evolving hypersurface for all $t \in [0,T]$. Then the following identity holds:
\begin{equation}
 \begin{aligned}
&\ipd{\vec{\mathscr{V}}\cdot\nabla_s\varkappa,~\chi}_{\Gamma} + \frac{1}{2}\ipd{\nabla_s\cdot\vec{\mathscr{V}},~[\varkappa-\bar{\varkappa}]\,\chi}_{\Gamma} =  \mathscr{A}_{\Gamma}(\vec{\mathscr{V}}, \varkappa-\bar{\varkappa},~\chi) - \frac{1}{2}\ipd{\vec{\mathscr{V}}\cdot\vec{\nu},~[\varkappa-\bar{\varkappa}]\,\varkappa\,\chi}_{\Gamma}, \label{eq:anti}
	\end{aligned}   
\end{equation}
for all $\chi\in H^1(\Gamma)$, where $\mathscr{A}_{\Gamma}$ is the antisymmetric term defined via
\begin{equation} \label{eq:antisym}
\mathscr{A}_{\Gamma}(\vec\eta, u, v) = \frac{1}{2}\bigl(\vec\eta\cdot\nabla_s u,~v\bigr)_{\Gamma} - \frac{1}{2}\bigl(\vec\eta\cdot\nabla_s v,~u\bigr)_{\Gamma}.
\end{equation}
\end{lem}

The weak formulation for \eqref{eq:newpde} can now be stated as follows. Initially, we are given the surface $\Gamma(0)$ and the curvature $\varkappa(\cdot,0)\in H^1(\Gamma(0))$. Then for each $t\in(0,T]$, we seek $\Gamma(t)=\vec x(\Upsilon, t)$ with $\mathscr{\vv V}(\cdot, t)\in [H^1(\Gamma)]^d$, $\left(\mathscr{V}, \varkappa, \kappa\right) \in [H^{1}(\Gamma)]^3$, and $\left(\lambda(t), \mu(t)\right) \in \mathbb{R}^2$ such that
\begin{subequations}\label{eqn:weak}
\begin{align}
&\ipd{\mathscr{V}, \varphi}_{\Gamma} - \ipd{\nabla_{s} \varkappa, \nabla_{s} \varphi}_{\Gamma} + \ipd{[\varkappa-\bar{\varkappa}]\left|\nabla_{s} \vec{\nu}\right|^{2}, \varphi}_{\Gamma} \nn \\
&\qquad - \frac{1}{2}\ipd{[\varkappa-\bar{\varkappa}]^{2} \varkappa, \varphi}_{\Gamma} - \lambda\ipd{1, \varphi}_{\Gamma} - \mu\ipd{\kappa, \varphi}_{\Gamma} = 0  \qquad \forall \varphi \in H^{1}(\Gamma), \label{eq:weak1}\\[0.6em]
\label{eq:weak2}
&\ipd{\partial_{t}^{\circ} \varkappa, \chi}_{\Gamma} + \frac{1}{2}\ipd{\nabla_{s} \cdot \vec{\mathscr{V}}, [\varkappa-\bar{\varkappa}] \chi}_{\Gamma} - \mathscr{A}_{\Gamma}(\vec{\mathscr{V}}, \varkappa-\bar{\varkappa}, \chi) + \ipd{\nabla_{s} \mathscr{V}, \nabla_{s} \chi}_{\Gamma} \nn\\
&\qquad - \ipd{\mathscr{V}\left|\nabla_{s} \vec{\nu}\right|^{2}, \chi}_{\Gamma} + \frac{1}{2}\ipd{\mathscr{V}, [\varkappa-\bar{\varkappa}] \varkappa \chi}_{\Gamma} = 0  \qquad \forall \chi \in H^{1}(\Gamma), \\[0.6em]
\label{eq:weak3}
&\qquad \ipd{\mathscr{V}, 1}_{\Gamma} = 0,\\[0.6em]
\label{eq:weak4}
&\qquad \ipd{\mathscr{V}, \kappa}_{\Gamma} = 0,\\[0.6em]
\label{eq:weak5}
&\ipd{\mathscr{\vv V}\cdot \vec{\nu}, \xi}_{\Gamma} - \ipd{\mathscr{V}, \xi}_{\Gamma} = 0\qquad \forall \xi \in H^{1}(\Gamma), \\
\label{eq:weak6}
&\ipd{\kappa\,\vec{\nu}, \vec{\eta}}_{\Gamma} + \ipd{\nabla_{s} \vec\id, \nabla_{s} \vec{\eta}}_{\Gamma} = 0\qquad \forall \vec{\eta} \in \left[H^{1}(\Gamma)\right]^d,
\end{align}
\end{subequations}

The following theorem shows that the weak formulation satisfies the geometric structures of the flow.
\begin{thm}
The weak solution of \eqref{eqn:weak} satisfies the energy law      
\begin{equation}\label{eq:energylawweak}
    \frac{\mathrm{d}}{\mathrm{d}t} E_{\bkap}(\Gamma(t), \varkappa(t))+\bigl(\mathscr{V}, \mathscr{V}\bigr)_{\Gamma(t)}=0.
\end{equation} 
Moreover, the enclosed volume and surface area are conserved
\begin{equation}
   \vol(\Gamma(t))=\vol(\Gamma(0)), \quad |\Gamma(t)|=|\Gamma(0)|,\qquad  t \geq 0.
\end{equation}
\end{thm}

\begin{proof} 
Using \eqref{eq:transport} gives
\begin{equation}\label{eq:energy_transport_identity}
    \frac{1}{2}\ddt\ipd{[\varkappa-\bkap]^2,~1}_{\Gt} = \ipd{\partial_t^\circ\varkappa,~\varkappa-\bkap}_{\Gt} +\frac{1}{2}\ipd{\nabs\cdot\mathscr{\vv V},~(\varkappa-\bkap)^2}_{\Gt}.
\end{equation}
Now choosing $\varphi = \mathscr{V}$ in \eqref{eq:weak1} and $\chi=\varkappa-\bkap$ in \eqref{eq:weak2}, recalling \eqref{eq:weak5}--\eqref{eq:weak6}, as well as the antisymmetric term in \eqref{eq:antisym}, we obtain
\begin{equation}
\bigl(\mathscr{V},~\mathscr{V}\bigr)_{\Gt} + \bigl(\partial_t^\circ\varkappa,~\varkappa-\bkap\bigr)_{\Gt} +\frac{1}{2}\Bigl(\nabs\cdot\mathscr{\vv V},~(\varkappa-\bkap)^2\Bigr)_{\Gt}=0,
\end{equation}
which, together with \eqref{eq:energy_transport_identity}, proves \eqref{eq:energylawweak}. 

It follows from the transport theorem 
\begin{equation}\label{eq:transp}
\ddt\vol(\Gamma(t)) =\ipd{\mathscr{V}, 1}_{\Gamma} = 0, 
\end{equation}
recalling \eqref{eq:weak3}. Using \eqref{eq:transport} again and the identity $\nabs\mathscr{\vv V} = \nabs\vec\id: \nabs\mathscr{\vv V}$, we have 
\begin{equation*}
\ddt|\Gamma(t)| = \ipd{\nabs\vec\id, \nabs\mathscr{\vv V}}_{\Gamma}.
\end{equation*}
We next set $\xi = \kappa$ in \eqref{eq:weak5} and $\vec\eta = \mathscr{\vv V}$ in \eqref{eq:weak6}, and recall \eqref{eq:weak4} to obtain that
\begin{equation}
\ddt|\Gamma(t)| = \ipd{\nabs\vec\id, \nabs\mathscr{\vv V}}_{\Gamma} = \ipd{\mathscr{V},~\kappa}_{\Gamma} = 0, 
\end{equation}
which implies the surface area preservation. 
\end{proof}

\begin{rem}
The proof also explains why the geometric curvature $\kappa$ is used to
  enforce surface area preservation, whereas the evolution curvature is used for the energy-stability estimate. 
\end{rem}

\section{Parametric finite element approximations}\label{sec:pfem}

We employ a uniform partition of the time interval $[0,T]=\cup_{m=1}^M[t_{m-1}, t_m]$, with $t_m = m\ttau$ and time step size $\ttau = \frac{T}{M}$. The evolving closed hypersurface $\Gamma(t_m)\subset\bR^d$ ($d=2,3$) is approximated by a $(d-1)$-dimensional polyhedral surface $\Gamma^m$:
\begin{equation}
 	\Gamma^m := \bigcup_{j=1}^{J}\overline{\sigma_j^m}, \quad \text{with} \quad \mT^m=\{\sigma_j^m\}_{j=1}^J \quad \text{and} \quad \mQ^m=\{\vec q_k^m\}_{k=1}^{K}, \label{eq:GammaD}
\end{equation}
where $\mT^m$ is a collection of mutually disjoint open $(d-1)$-simplices (line segments for $d=2$ and triangles for $d=3$), and $\mQ^m$ is the set of globally labeled vertices. 

For each element $\sigma_j^m \in \mT^m$, let $\{\vec q_{j_k}^m\}_{k=1}^d$ be its vertices, ordered to ensure a consistent outward orientation. To unify the geometric description for $d\in\{2,3\}$, we define the unnormalized outward normal vector $\vec{N}(\sigma_j^m)$ as follows:
\begin{equation}\label{eq:J_vector}
    \vec{N}(\sigma_j^m) := 
    \begin{cases}
        (\vec q_{j_2}^m - \vec q_{j_1}^m)^\bot, & \text{if } d=2, \\
        (\vec q_{j_2}^m - \vec q_{j_1}^m) \times (\vec q_{j_3}^m - \vec q_{j_1}^m), & \text{if } d=3,
    \end{cases}
\end{equation}
where $(\cdot)^\perp$ represents a clockwise rotation by $\frac{\pi}{2}$ in $\bR^2$. Accordingly, the discrete unit normals $\vec\nu^m$ on $\Gamma^m$ are computed element-wise by
\begin{equation}\label{eq:vG}
    \vec{\nu}^m|_{\sigma^{m}_j} := \frac{\vec{N}(\sigma_j^m)}{|\vec{N}(\sigma_j^m)|},\qquad\forall\sigma_j^m\in\mT^m.
\end{equation}
We also follow \cite{BGN08parametric, BGN08willmore} and introduce the vertex normal $\vec\omega^{m}\in [\bV^h(\Gm)]^d$ as the mass-lumped $L^2$-projection of the face normal $\vec\nu^m$ onto $[\bV^h(\Gm)]^d$:
\begin{equation} \label{eq:nuhomegah}
    \ipd{\vec\omega^m, \vec\eta^h}_{\Gm}^h = \ipd{\vec\nu^m, \vec\eta^h}_{\Gm} \quad \forall\vec\eta^h\in [\bV^h(\Gm)]^d.
\end{equation}
It naturally follows that $\ipd{\chi\,\vec\omega^m, \vec\eta^h}_{\Gm}^h = \ipd{\chi\,\vec\nu^m, \vec\eta^h}_{\Gm}^h$ holds for any $\chi\in \bV^h(\Gm)$ and $\vec\eta^h\in [\bV^h(\Gm)]^d$.

Associated with $\Gamma^m$, we introduce the finite element space 
\[\bV^h(\Gm):=\bigl\{\varphi\in C(\Gm): \varphi|_{\sigma}\quad\mbox{is affine}\quad\forall\sigma\in\mT^m\bigr\}.\]
To approximate the inner product $\ipd{\cdot,\cdot}_{\Gm}$, we introduce the mass-lumped approximation over the current polyhedral surface $\Gamma^m$ via 
\begin{equation}
\ipd{u, v}_{\Gm}^h:=  
\frac{1}{d}\sum_{j=1}^{J}|\sigma^{m}_j|
\sum_{k=0}^{d-1} 
\underset{\sigma^{m}_j\ni \vec{p}\to \vec{q}^{m}_{j_k}}{\lim}\, 
(u\cdot v)(
\vec{p}),\label{eq:tprule}
\end{equation}
where $u,v$ are piecewise continuous, with possible jumps
across the edges of $\sigma\in\mT^m$, and 
$|\sigma^{m}_j|= \frac{1}{(d-1)!}\,|\vec{N}(\sigma_j^{m})|$ 
is the measure of $\sigma^{m}_j$.

We next present our structure-preserving parametric finite element approximation for the Helfrich flow, which consists of two parts: the approximation of the gradient-flow structure and the application of the computed normal velocity together with the BGN tangential velocity.

\subsection{Gradient flow approximations} 

To construct an unconditionally energy-stable fully discrete scheme for the Helfrich flow, we adopt the arbitrary Lagrangian-Eulerian (ALE) parametric finite element framework recently introduced in \cite{GNZ26}. 

Given the polyhedral surfaces $\{\Gamma^\ell\}_{\ell\leq m}$, we first define the discrete vertex velocity $\mathscr{\vv V}^m\in[\bV^h(\Gm)]^d$ and the discrete pullback mapping $\vec\Phi^m\in[\bV^h(\Gm)]^d$ as
\begin{equation} \label{eq:Vm_and_mapping}
    \mathscr{\vv V}^m(\vec q_k^m) = \frac{\vec q_k^{m} - \vec q_k^{m-1}}{\ttau}, \qquad \vec\Phi^m = \vec\id|_{\Gm} - \ttau\,\mathscr{\vv V}^m,
\end{equation}
so that $\vec\Phi^m(\Gamma^m) = \Gamma^{m-1}$. Then we have the following change-of-variables formula 
\begin{equation} \label{eq:change_of_variables}
\int_{\Gamma^{m-1}} f \,\dH^{d-1} = \int_{\Gamma^m} f \circ \vec{\Phi}^m \mathcal{J}^m \,\dH^{d-1} \qquad \forall f \in L^1(\Gamma^{m-1}),
\end{equation}
where $\mathcal{J}^m=\sqrt{{\rm det}(\mathcal{L}^m)}$ is the surface Jacobian determinant of the map $\vec\Phi^m$, and $\mathcal{L}^m =[\nabs\vec\Phi^m(\vec z)]^T\nabs\vec\Phi^m(\vec z) $ is a linear operator on the tangent space $T_{\vec z}\Gm$:
\[\mathcal{L}^m(\vec z): T_{\vec z}\Gm\mapsto T_{\vec z}\Gm.\]

Denote by $\varkappa^m$, $\mathscr{V}^m$, and $\kappa^m$ the numerical approximations of $\varkappa(\cdot,t_m)$, $\mathscr{V}(\cdot,t_m)$, and $\kappa(\cdot,t_m)$, respectively, defined on $\Gamma^{m-1}$. To construct a linear, unconditionally stable scheme, we employ a special treatment of the
  first two terms in \eqref{eq:weak2}, following \cite{GNZ26}. Recalling the definition of the velocity in \eqref{eq:velocity}, we obtain the following lemma, which provides a consistent temporal discretization of these two terms.

\begin{lem}[Metric-bundled time derivative]\label{lem:metric_bundling}
Let $\varkappa_{\Gamma^m}^m = \varkappa^m \circ \vec{\Phi}^m$ be the pullback of the previous mean curvature onto the current polyhedral mesh $\Gamma^m$ and assume $\nabs\mathscr{\vv V}$ remains uniformly bounded during the evolution. Then the following consistency relation holds:
\begin{align}\label{eq:metric_bundling}
    &\left(\frac{(\varkappa^{m+1}-\bkap) - (\varkappa_{\Gamma^m}^m-\bkap) \sqrt{\mathcal{J}^m}}{\ttau}, \chi^h\right)_{\Gamma^m} \nn\\
    &= \left(\frac{\varkappa^{m+1}-\varkappa_{\Gamma^m}^m}{\ttau}, \chi^h\right)_{\Gamma^m}+ \frac{1}{2}\left(\nabla_s \cdot \mathscr{\vv V}^m, (\varkappa_{\Gamma^m}^m-\bkap)\chi^h\right)_{\Gamma^m} + \mathcal{O}(\ttau),\quad \forall\chi^h \in \bV^h(\Gamma^m).
\end{align}
\end{lem}

\begin{proof}
We first recall \cite[Lemma~4.1]{GNZ26} that 
for a sufficiently small time step $\ttau$, it holds that
\begin{equation}\label{eq:Jmexp}
	\sqrt{\mathcal{J}^m} = 1 - \frac{1}{2}\ttau\,\nabs\cdot \mathscr{\vv V}^m + \mathcal{O}(\ttau^2).
\end{equation}
Now we substitute \eqref{eq:Jmexp} into the left-hand side of \eqref{eq:metric_bundling} and apply the Taylor expansion to obtain 
\begin{align*}
    &\frac{(\varkappa^{m+1}-\bkap) - (\varkappa_{\Gamma^m}^m-\bkap) \sqrt{\mathcal{J}^m}}{\ttau} \\
    &\qquad = \frac{(\varkappa^{m+1}-\bkap) - (\varkappa_{\Gamma^m}^m-\bkap)\left[ 1 - \frac{1}{2}\ttau \nabla_s \cdot \mathscr{\vv V}^m + \mathcal{O}(\ttau^2) \right]}{\ttau} \\
    &\qquad = \frac{\varkappa^{m+1} - \varkappa_{\Gamma^m}^m}{\ttau} + \frac{1}{2}(\varkappa_{\Gamma^m}^m-\bkap) \nabla_s \cdot \mathscr{\vv V}^m + \mathcal{O}(\ttau).
\end{align*}
Multiplying by the test function $\chi^h$ and integrating over $\Gamma^m$ yields the desired result \eqref{eq:metric_bundling}. 
\end{proof}

We are now ready to present a linear and unconditionally stable scheme for approximating the gradient-flow structure of the Helfrich flow. Given an initial admissible polyhedral surface $\Gamma^0$, we set $\Gamma^{-1}=\Gamma^0$ and $\mathscr{\vv V}^0 = \vec{0}$. We also assume we have the initial curvatures $(\varkappa^0, \kappa^0)\in[\bV^h(\Gm)]^2$. For $m \geq 0$, we first introduce an explicit approximation of $|\nabs\vec\nu|^2$ by setting
\begin{equation} \label{eq:Wm}
    \mathcal{W}^m = | \nabs\vec v^m|^2\qquad\mbox{with}\quad \vec v^m(\vec q) = \frac{\vec\omega^m(\vec q)}{|\vec\omega^m(\vec q)|},
\end{equation}
where $\vec\omega^m$ is the vertex normal defined in \eqref{eq:nuhomegah}, and $\vec v^m$ is the normalized vertex normal. Then we find $\left(\mathscr{V}^{m+1}, \varkappa^{m+1}\right) \in [\bV^h(\Gamma^m)]^2$ and the Lagrange multipliers $\left(\lambda^{m+1},\mu^{m+1} \right)\in\bR^2$ such that
\begin{subequations}
\label{eqn:fd}
\begin{align}
    &\ipd{\mathscr{V}^{m+1}, \varphi^h}_{{\Gm}} - \bigl(\nabs\varkappa^{m+1}, \nabs\varphi^h\bigr)_{\Gm} + \ipd{\mathcal{W}^m\,[\varkappa^{m+1}-\bkap], \varphi^h}_{\Gm} \nn\\
    &\qquad - \frac{1}{2}\ipd{[\varkappa^m_{{\Gamma^m}}-\bkap]\,\varkappa_{\Gamma^m}^m\,[\varkappa^{m+1}-\bkap], \varphi^h}_{\Gm} \nn\\
    &\qquad - \lambda^{m+1}\ipd{1, \varphi^{h}}_{\Gamma^{m}} - \mu^{m+1}\ipd{\kappa^{m}, \varphi^{h}}_{\Gamma^{m}} = 0 \qquad\forall\varphi^h\in\bV^h(\Gm),\label{eq:fd1}\\[0.5em]
    &\ipd{\frac{\varkappa^{m+1}-\bkap - (\varkappa_{\Gamma^m}^m-\bkap)\,\sqrt{\mathcal{J}^m}}{\ttau}, \chi^h}_{\Gm} -\mathscr{A}_{\Gm}(\mathscr{\vv V}^m, \varkappa^{m+1}-\bkap, \chi^h) \nn\\
    &\qquad + \ipd{\nabs\mathscr{V}^{m+1}, \nabs\chi^h}_{\Gm} - \ipd{\mathcal{W}^m\,\mathscr{V}^{m+1}, \chi^h}_{\Gm} \nn\\
    &\qquad + \frac{1}{2}\ipd{\mathscr{V}^{m+1}, [\varkappa^m_{\Gamma^m}-\bkap]\,\varkappa_{\Gamma^m}^m\,\chi^h}_{\Gm} = 0\qquad\forall\chi^h\in\bV^h(\Gm);\label{eq:fd2} \\[0.5em]
    &  \ipd{\mathscr{V}^{m+1}, 1}_{\Gamma^{m}} = 0, \label{eq:fd3}\\
    & \ipd{\mathscr{V}^{m+1} \kappa^{m}, 1}_{\Gamma^{m}} = 0. \label{eq:fd4}
\end{align} 
\end{subequations}

We have the following theorem, which shows that the linear system \eqref{eqn:fd} admits a unique solution. 

\begin{thm}[well-posedness]
Assume that 
\begin{enumerate}[label=$(\mathbf{A \arabic*})$, ref=$\mathbf{A \arabic*}$] 
\item \label{assumpI}  The polyhedral surface satisfies 
\[|\sigma|>0\qquad\forall\sigma\in\mT^m;\]
\item \label{assumpII} The geometric curvature is not a constant, i.e., $\kappa^m\not\equiv c$ for any constant $c$.
\end{enumerate}
Then the linear system in \eqref{eqn:fd} admits a unique solution
\[
(\mathscr{V}^{m+1}, \varkappa^{m+1}, \lambda^{m+1}, \mu^{m+1})
\in \bV^h(\Gamma^m) \times \bV^h(\Gamma^m) \times \bR \times \bR.
\]
\end{thm}
\begin{proof}
Since the linear system is finite-dimensional and the number of unknowns matches the number of equations, the Fredholm alternative implies that it suffices to prove that the corresponding homogeneous system admits only the trivial zero solution. We thus consider the corresponding homogeneous system, which is given by finding $(\mathscr{V}, \varkappa)\in[\bV^h(\Gm)]^2$ and $(\lambda, \mu)\in\bR^2$ such that
\begin{subequations}\label{eq:homo_stage1}
\begin{align}
    &\ipd{\mathscr{V}, \varphi^h}_{{\Gm}} - \ipd{\nabs\varkappa, \nabs\varphi^h}_{\Gm} + \ipd{\mathcal{W}^m\,\varkappa, \varphi^h}_{\Gm} \nn\\
    &\qquad - \frac{1}{2}\ipd{(\varkappa^m_{{\Gamma^m}})^2\,\varkappa, \varphi^h}_{\Gm} - \lambda \ipd{1, \varphi^h}_{\Gamma^m} - \mu \ipd{\kappa^m, \varphi^h}_{\Gamma^m} = 0\qquad\forall\varphi^h\in\bV^h(\Gm); \label{eq:homo1} \\[0.5em]
    &\ipd{\frac{\varkappa}{\ttau}, \chi^h}_{\Gm} - \mathscr{A}_{\Gm}(\mathscr{\vv V}^m, \varkappa, \chi^h) + \ipd{\nabs\mathscr{V}, \nabs\chi^h}_{\Gm} \nn\\
    &\qquad - \ipd{\mathcal{W}^m\,\mathscr{V}, \chi^h}_{\Gm} + \frac{1}{2}\ipd{\mathscr{V}, (\varkappa^m_{\Gamma^m})^2\,\chi^h}_{\Gm} = 0\qquad\forall\chi^h\in\bV^h(\Gm);\label{eq:homo2} \\[0.5em]
    &\ipd{\mathscr{V}, 1}_{\Gamma^m} = 0;\label{eq:homo3} \\[0.5em]
    &\ipd{\mathscr{V}, \kappa^m}_{\Gamma^m} = 0.\label{eq:homo4}
\end{align}
\end{subequations}
We then choose $\varphi^h = \ttau\mathscr{V}$ in \eqref{eq:homo1}, $\chi^h = \varkappa$ in \eqref{eq:homo2}, multiply \eqref{eq:homo2} by $\ttau$, and add the two equations. The antisymmetric term vanishes by skew-symmetry, and the Lagrange multiplier terms vanish due to \eqref{eq:homo3}--\eqref{eq:homo4}. This gives
\begin{equation*}
    \ttau\ipd{\mathscr{V}, ~\mathscr{V}}_{\Gm} + \ipd{\varkappa,~\varkappa}_{\Gamma^m} = 0.
\end{equation*}
This immediately implies $\mathscr{V} \equiv 0$ and $\varkappa \equiv 0$ on $\Gamma^m$.

Substituting $\mathscr{V} = 0$ and $\varkappa = 0$ back into \eqref{eq:homo1} yields the orthogonality condition for the multipliers:
\begin{equation}\label{eq:multiplier_system}
    \lambda \ipd{1, \varphi^h}_{\Gamma^m} + \mu \ipd{\kappa^m, \varphi^h}_{\Gamma^m} = 0 \quad \forall \varphi^h \in \bV^h(\Gamma^m).
\end{equation}
Since $\Gamma^m$ is closed, choosing $\varphi^h = 1$ and $\varphi^h = \kappa^m$ leads to the $2 \times 2$ Gram matrix system:
\begin{equation}\label{eq:gram_matrix}
    \begin{pmatrix}
        \ipd{1, 1}_{\Gamma^m} & \ipd{1, \kappa^m}_{\Gamma^m} \\[0.5em]
        \ipd{\kappa^m, 1}_{\Gamma^m} & \ipd{\kappa^m, \kappa^m}_{\Gamma^m}
    \end{pmatrix}
    \begin{pmatrix}
        \lambda \\ \mu
    \end{pmatrix}
    =
    \begin{pmatrix}
        0 \\ 0
    \end{pmatrix}.
\end{equation}
By assumption~\ref{assumpII}, the functions $1$ and $\kappa^m$ are linearly independent in $L^2(\Gamma^m)$. Hence the Gram matrix is invertible, which leads to $\lambda = 0$ and $\mu = 0$. 

This shows that the corresponding homogeneous system has only the trivial zero solution. Thus the linear system \eqref{eqn:fd} has a unique solution. 
\end{proof}

\begin{rem}[The degenerate case]
If assumption~\ref{assumpII} is violated, i.e., if $\kappa^m \equiv c$ for a constant $c$, then the constraints in \eqref{eq:fd3} and \eqref{eq:fd4} become linearly dependent. To remove this degeneracy, we may simply set $\mu=0$, so that the resulting reduced system remains uniquely solvable.
\end{rem}

The next theorem establishes the unconditional energy stability of the method. We define the discrete Helfrich bending energy at time $t_m$ by
\begin{equation}\label{eq:discrete_energy}
\mathcal{E}(\Gamma^{m-1}, \varkappa^m)
  := \frac{1}{2} \int_{\Gamma^{m-1}} (\varkappa^m - \bar{\varkappa})^2 \,\dH^{d-1}.
\end{equation}
  
\begin{thm}[unconditional energy stability]\label{thm:stability}

Let $(\mathscr{V}^{m+1}, \varkappa^{m+1})\in[\bV^h(\Gm)]^2$ be a solution to \eqref{eqn:fd}. Then for any time step size $\ttau > 0$, the following energy dissipation inequality holds:
\begin{equation} \label{eq:energy_dissipation}
\mathcal{E}(\Gamma^{m}, \varkappa^{m+1}) + \ttau\norm{\mathscr{V}^{m+1}}^2_{\Gm} \le \mathcal{E}(\Gamma^{m-1}, \varkappa^m),
\end{equation}
where $\norm{\cdot}_{\Gm}$ is the norm induced by the inner product $\ipd{\cdot, \cdot}_{\Gm}$.
\end{thm}

\begin{proof}
We set $\varphi^h = \ttau\mathscr{V}^{m+1}$ in \eqref{eq:fd1}, choose $\chi^h = \varkappa^{m+1} - \bar{\varkappa}$ in \eqref{eq:fd2}, multiply \eqref{eq:fd2} by $\ttau$, and combine the two equations to obtain, after recalling the constraints \eqref{eq:fd3} and \eqref{eq:fd4},
\begin{equation}\label{eq:es1}
\ttau \ipd{\mathscr{V}^{m+1}, \mathscr{V}^{m+1}}_{\Gm} + \ipd{(\varkappa^{m+1}-\bar{\varkappa}) - (\varkappa_{\Gamma^m}^{m}-\bar{\varkappa}) \sqrt{\mathcal{J}^{m}}, \, \varkappa^{m+1}-\bar{\varkappa}}_{\Gamma^{m}} = 0.
\end{equation}

Applying the inequality $a(a-b) \ge \frac{1}{2}(a^2 - b^2)$ to the second term in \eqref{eq:es1}, we obtain the lower bound:
\begin{align*} 
&\left( (\varkappa^{m+1}-\bar{\varkappa}) - (\varkappa_{\Gamma^m}^{m}-\bar{\varkappa}) \sqrt{\mathcal{J}^{m}}, \, \varkappa^{m+1}-\bar{\varkappa} \right)_{\Gamma^{m}} \\
&\qquad \ge \frac{1}{2} \norm{\varkappa^{m+1}-\bar{\varkappa}}_{\Gamma^m}^2 - \frac{1}{2} \norm{(\varkappa_{\Gamma^m}^m-\bar{\varkappa})\sqrt{\mathcal{J}^m}}_{\Gamma^m}^2 \\
&\qquad = \mathcal{E}(\Gm, \varkappa^{m+1}) - \frac{1}{2} \int_{\Gamma^m} (\varkappa_{\Gamma^m}^m-\bar{\varkappa})^2 \mathcal{J}^m \,\dH^{d-1}.
\end{align*}
Finally, by the change-of-variables formula \eqref{eq:change_of_variables}, the last integral on $\Gamma^m$ pulls back exactly to the energy evaluated on the previous mesh $\Gamma^{m-1}$. Thus,
\begin{equation}
\ttau \|\mathscr{V}^{m+1}\|_{\Gamma^m}^2 + \mathcal{E}(\Gm, \varkappa^{m+1}) - \mathcal{E}(\Gamma^{m-1}, \varkappa^m) \le 0,
\end{equation}
which gives the desired result in \eqref{eq:energy_dissipation}.
\end{proof}

It is noteworthy that the stability results in both \eqref{eq:energylawweak} and \eqref{eq:energy_dissipation} control only the normal component of the velocity, not the full velocity. This is consistent with the fact that the gradient-flow structure in \eqref{eq:3d_velocity} is determined solely by the normal velocity $\mathscr{V}$ on $\Gt$.

We refer to \eqref{eq:fd3} and \eqref{eq:fd4} as soft constraints, since they involve only the normal velocity. Exact preservation of volume and surface
area, however, depends crucially on how the polyhedral mesh is updated, as discussed in the following subsection. 

\subsection{Geometric mesh movement}

The enclosed volume and surface area of the polyhedral surface $\Gm$ can be written as 
\begin{equation}\label{eq:discrete_vol_area} 
\begin{aligned} 
    \vol(\Gm) &:= \frac1d \int_{\Gamma^m} \vec\id \cdot \vec\nu^m \, \dH^{d-1} = \frac{1}{d!\,d}\sum_{j=1}^J \sum_{k=1}^d \vec q_{j_k}^m \cdot \vec{N}(\sigma_j^m), \\ 
    |\Gm| &:= \int_{\Gamma^m} 1\,\dH^{d-1} = \sum_{j=1}^J |\sigma_j^m|=\frac{1}{(d-1)!}\sum_{j=1}^J|\vec{N}(\sigma_j^m)|. 
\end{aligned} 
\end{equation} 

Given the normal velocity $\mathscr{V}^{m+1}$ solved from \eqref{eqn:fd}, we then combine it with the BGN tangential velocity to update the polyhedral surface mesh. 

To enable exact volume preservation, we follow the work in \cite{BZ21SPFEM} and introduce the time-weighted interface normals $\vec{\nu}^{m+\frac{1}{2}}$ element-wise on $\Gm$ as:
\begin{equation}\label{eq:J_eff}
    \vec\nu^{m+\frac{1}{2}}|_{\sigma_j^m} :=
    \begin{cases}
        \frac{1}{2\,|\vec{N}(\sigma_j^m)|} \bigl[ \vec{N}(\sigma_j^m) + \vec{N}(\sigma_j^{m+1}) \bigr], & \text{if } d=2, \\[1.5ex]
        \frac{1}{6\,|\vec{N}(\sigma_j^m)|} \bigl[ \vec{N}(\sigma_j^m) + 4\,\vec{N}(\sigma_j^{m+\frac12}) + \vec{N}(\sigma_j^{m+1}) \bigr], & \text{if } d=3.
    \end{cases}
\end{equation}
Here, for $d=3$, $\sigma_j^{m+\frac12}$ denotes the intermediate triangle formed by the vertices 
\[\vec q_{j_k}^{m+\frac12} := \frac12(\vec q_{j_k}^m+\vec q_{j_k}^{m+1})\quad\mbox{for}\quad k=1,2,3.\]

Recalling the transport theorem in \eqref{eq:transp}:
\[\frac{\mathrm{d}}{\mathrm{d} t}\vol(\Gamma(t))  = \int_{\Gamma(t)} \mathscr{V} \dH^{d-1},\]
we then have its discrete analogue using the time-weighted normals in \eqref{eq:J_eff}. 
\begin{lem}[discrete volume identity] \label{lem:volume_conservation}
If $\vec X^{m+1}\in[\bV^h(\Gm)]^d$ with $\Gamma^{m+1}=\vec X^{m+1}(\Gm)$, then the following identity holds:
\begin{equation} \label{eq:exact_vol_diff}
    \vol(\Gamma^{m+1})- \vol(\Gamma^m) = \ipd{\vec{X}^{m+1}-\vec\id,~\vec{\nu}^{m+\frac{1}{2}}}^h_{\Gamma^m}.
\end{equation}
\end{lem}
\begin{proof}
The proof can be found in \cite[Theorems~2.1 and~3.1]{BZ21SPFEM}.
\end{proof}

To preserve the surface area, we introduce an additional Lagrange multiplier that corrects the mesh update implicitly without influencing volume preservation. 

Now we are ready to present the BGN-type discretization for the movement of the polyhedral mesh. For each $m\geq 0$, we are given $\mathscr{V}^{m+1}$. We then seek $\Gamma^{m+1}=\vec X^{m+1}(\Gm)$ with $\vec X^{m+1}\in[\bV^h(\Gm)]^d$, $\kappa^{m+1}\in \bV^h(\Gamma^m)$, and the area-correction multiplier $\alpha^{m+1} \in \mathbb{R}$ such that
\begin{subequations}\label{eqn:fdBGN}
\begin{align}
&\ipd{\frac{\vec{X}^{m+1} - \vec\id }{\ttau} \cdot \vec{\nu}^{m+\frac{1}{2}}, \xi^h}^h_{\Gamma^m} = \ipd{\mathscr{V}^{m+1}, \xi^h }_{\Gamma^m}  + \alpha^{m+1} \ipd{\kappa^m - \bar{\kappa}^m, \xi^h}^h_{\Gamma^m};  \label{eq:fdBGN1} \\[0.5em]
&\ipd{\kappa^{m+1} \vec{\nu}^m, \vec{\eta}^h}^h_{\Gamma^m} + \ipd{\nabs \vec{X}^{m+1}, \nabs \vec{\eta}^h }_{\Gamma^m} = 0; \label{eq:fdBGN3} \\[0.5em]
& \qquad |\Gamma^{m+1}| = |\Gamma^0|;\label{eq:fdBGN2}
\end{align}
\end{subequations} 
for $\left(\xi^h, \vec\eta^h\right)\in\bV^h(\Gm)\times [\bV^h(\Gm)]^d$, where $\bar{\kappa}^m=\frac{(\kappa^{m}, 1)^h_{\Gamma^m}}{(1,1)^h_{\Gamma^m}}$. 

Here, the term $\alpha^{m+1}(\kappa^m-\bar{\kappa}^m)$ serves as the Lagrange multiplier term for the hard area constraint \eqref{eq:fdBGN2}. Because it has zero mean, the volume preservation property remains unchanged. Note that \eqref{eqn:fdBGN} also yields the geometric curvature $\kappa^{m+1}$, which will be used in both the soft and hard constraints for surface area preservation at the next time step.

We have the following theorem for the exact preservation of the discrete volume and surface area. 

\begin{thm}[geometric preservation] Let $(\mathscr{V}^{m+1}, \varkappa^{m+1}, \lambda^{m+1}, \mu^{m+1})\in [\bV^h(\Gm)]^2\times \bR^2$ be a solution of
  \eqref{eqn:fd}. Let $(\vec X^{m+1}, \kappa^{m+1}, \alpha^{m+1})
  \in [\bV^h(\Gm)]^d \times \bV^h(\Gm) \times \bR$ be a solution of \eqref{eqn:fdBGN}, and set $\Gamma^{m+1}=\vec X^{m+1}(\Gm)$.
Then
\begin{subequations}\label{eqn:geopre}
\begin{align}\label{eq:geovolume}
\vol(\Gamma^{m+1}) &=\vol(\Gm),\qquad m\geq 0;\\
|\Gamma^{m+1}| & = |\Gamma^m|,\qquad m\geq 0. 
\label{eq:geoarea}
\end{align}
\end{subequations}
\end{thm}
\begin{proof}
The surface area preservation \eqref{eq:geoarea} follows straightforwardly from the hard constraint \eqref{eq:fdBGN2}.

For the volume preservation, we set $\xi^h = \ttau$ in \eqref{eq:fdBGN1}, and recall the soft constraint in \eqref{eq:fd3}. This gives
\begin{equation}
\ipd{(\vec X^{m+1}-\vec\id)\cdot\vec\nu^{m+\frac{1}{2}},~1}_{\Gm}^h = \ttau\ipd{\mathscr{V}^{m+1},~1}_{\Gm} = 0,
\end{equation}
which implies the volume preservation \eqref{eq:geovolume} using \eqref{eq:exact_vol_diff}.
\end{proof}

Note that \eqref{eqn:fdBGN} leads to a system of nonlinear equations. To solve the nonlinear system and enforce the global area constraint, we use a Newton--Raphson iteration strategy. At the discrete time $t_m$, for each $\ell\geq 0$, we set $(\vec{X}^{m+1,(\ell)}, \kappa^{m+1,(\ell)}, \alpha^{m+1,(\ell)})$ to be the current state at the $\ell$-th Newton iteration with $\Gamma^{m+1,(\ell)} = \vec X^{m+1,(\ell)}(\Gm)$. We then seek increments $(\delta\vec{X}, \delta\kappa, \delta\alpha)\in[\bV^h(\Gm)]^d\times \bV^h(\Gm)\times\bR$ such that $\vec{X}^{m+1,(\ell+1)} = \vec{X}^{m+1,(\ell)} + \delta\vec{X}$, with analogous updates for $\kappa$ and $\alpha$. The Newton system for the increments is given as follows:
\begin{subequations}\label{eq:newton_system}
\begin{align}
    &\ipd{\frac{\delta\vec X}{\ttau}\cdot \vec{\nu}^{m+\frac12,(\ell)} + \vec{A}^{(\ell)} \cdot \delta\vec X, \, \xi^h}^h_{\Gamma^m}  = -  \ipd{\frac{\vec{X}^{m+1,(\ell)} - \vec\id }{\ttau} \cdot \vec{\nu}^{m+\frac12,(\ell)}, \, \xi^h }^h_{\Gamma^m}\nn\\
    &\qquad\qquad\qquad+ \ipd{\mathscr{V}^{m+1}, \, \xi^h }_{\Gamma^m}  - [\alpha^{m+1,(\ell)}+\delta\alpha]\ipd{\kappa^m - \bar{\kappa}^m, \, \xi^h}^h_{\Gamma^m}, \label{eq:newton_1} \\[0.7em]
    &\ipd{\delta\kappa \, \vec{\nu}^m, \, \vec{\eta}^h }^h_{\Gamma^m} + \ipd{\nabs(\delta\vec{X}), \, \nabs \vec{\eta}^h}_{\Gamma^m} = - \ipd{\kappa^{m+1,(\ell)} \vec{\nu}^m, \, \vec{\eta}^h}^h_{\Gamma^m} - \ipd{\nabs \vec{X}^{m+1,(\ell)}, \, \nabs \vec{\eta}^h}_{\Gamma^m} , \label{eq:newton_2} \\[0.7em]
&\ipd{\nabs(\delta\vec{X}),~\nabs\vec X^{m+1,(\ell)}}_{\Gamma^{m+1,(\ell)}} = |\Gamma^0|- |\Gamma^{m+1,(\ell)}|, \label{eq:newton_3}
\end{align}
\end{subequations}
for all test functions $\left(\xi^h, \vec\eta^h\right)\in \bV^h(\Gamma^m)\times [\bV^h(\Gamma^m)]^d$. 

In \eqref{eq:newton_system}, the intermediate normal vector at the $\ell$-th iteration, denoted by $\vec{\nu}^{m+\frac12,(\ell)}$, is computed from \eqref{eq:J_eff} with the unknown future state $\vec{X}^{m+1}$ replaced by the current iterate $\vec{X}^{m+1,(\ell)}$. 

The vector field $\vec{A}^{(\ell)}$ is assembled directly from element-wise nodal contributions; see also \cite[(2.24),(3.26)]{BZ21SPFEM}. We illustrate this in the case $d=3$. Let $\sigma^m\in\mT^m$ be an element on $\Gamma^m$ with well-ordered vertices $\{\vec{q}_1^m, \vec{q}_2^m, \vec{q}_3^m\}$, and let
\[\{\vec{X}_1^{(\ell)}, \vec{X}_2^{(\ell)}, \vec{X}_3^{(\ell)}\}=\{\vec X^{m+1,(\ell)}(\vec q_1^m), \vec X^{m+1,(\ell)}(\vec q_2^m), \vec X^{m+1,(\ell)}(\vec q_3^m)\}\]
 be the corresponding vertices on $\Gamma^{m+1,(\ell)}$. The local contribution $\vec{a}_i(\sigma^m)$ to $\vec{A}^{(\ell)}$ at the $i$-th vertex of the element ($i=1,2,3$) is explicitly computed as:
\begin{equation} \label{eq:A_l_eval}
    \vec{a}_i(\sigma^m) = \frac{1}{6 |\vec{N}(\sigma^m)|} \vec{V}_{\sigma}^{(\ell)} \times \vec{g}_{jk}^{(\ell)},
\end{equation}
where $(i,j,k)$ is a cyclic permutation of $(1,2,3)$. Here, $\vec{V}_{\sigma}^{(\ell)} = \frac{1}{3\ttau} \sum_{r=1}^3 (\vec{X}_r^{(\ell)} - \vec{q}_r^m)$ is the average element velocity, and $\vec{g}_{jk}^{(\ell)} = (2\vec{X}_k^{(\ell)} + \vec{q}_k^m) - (2\vec{X}_j^{(\ell)} + \vec{q}_j^m)$. The global vector field $\vec{A}^{(\ell)}$ at any mesh vertex is then obtained by taking the arithmetic mean of the local contributions $\vec{a}_i(\sigma^m)$ from all adjacent elements sharing that vertex. Precisely, for each vertex $\vec q\in\mQ^m$, let
  \[
  \Lambda^m(\vec q):=\{\sigma\in\mT^m:\vec q\in\overline{\sigma}\}
  \]
  be the set of elements adjacent to $\vec q$. For each $\sigma\in\Lambda^m(\vec q)$, let $i_\sigma(\vec q)\in\{1,2,3\}$ denote the local index of $\vec q$ in $\sigma$, i.e., $\vec q=\vec q^m_{i_\sigma(\vec q)}(\sigma)$. We then define
  \begin{equation}
  \vec A^{(\ell)}(\vec q)
  :=
  \frac{1}{\#\Lambda^m(\vec q)}
  \sum_{\sigma\in\Lambda^m(\vec q)}
  \vec a_{i_\sigma(\vec q)}(\sigma).
  \end{equation}

The above iteration is terminated when the absolute $L^2$-norm of the displacement increment satisfies
 \[
 \|\delta\vec{X}\|_{\Gamma^m}\leq {\rm tol},
 \] 
 where ${\rm tol}$ is the chosen tolerance. Then the surface is updated for the next time step as $\Gamma^{m+1}=\vec X^{m+1,(\ell_{\max})}(\Gamma^m)$.

\subsection{Algorithmic variants and practical strategies}

The fully discrete schemes \eqref{eqn:fd} and \eqref{eqn:fdBGN} strike a balance between unconditional energy stability and exact global geometric preservation. In practice, however, the preferred choice may depend on the computational budget and the regularity of the initial data, with efficiency or robustness taking priority in different settings. We therefore present two practical variants of the proposed scheme.

\subsubsection{Variant I: Fully linear scheme without exact geometric correction}

To enable exact preservation of area and volume, we introduce nonlinear terms in \eqref{eqn:fdBGN}, which in turn require a Newton iteration. Alternatively, one may sacrifice machine-precision preservation of the volume and area, and instead employ the classical BGN-type linear scheme for the mesh update.

Specifically, we keep \eqref{eqn:fd} unchanged. For the mesh update in \eqref{eqn:fdBGN}, however, we instead find $(\vec X^{m+1}, \kappa^{m+1})\in [\bV^h(\Gm)]^d\times \bV^h(\Gm)$ such that
\begin{subequations}\label{eqn:linfdBGN}
\begin{align}
&\ipd{\frac{\vec{X}^{m+1} - \vec\id }{\ttau} \cdot \vec{\nu}^{m}, \xi^h}^h_{\Gamma^m} = \ipd{\mathscr{V}^{m+1}, \xi^h }_{\Gamma^m}\qquad\forall\xi^h\in\bV^h(\Gm);  \label{eq:linfdBGN1} \\[0.5em]
&\ipd{\kappa^{m+1} \vec{\nu}^m, \vec{\eta}^h}^h_{\Gamma^m} + \ipd{\nabs \vec{X}^{m+1}, \nabs \vec{\eta}^h }_{\Gamma^m} = 0\qquad\forall\vec\eta^h\in[\bV^h(\Gm)]^d.\label{eq:linfdBGN2} 
\end{align}
\end{subequations} 
This variant combines \eqref{eqn:fd} and \eqref{eqn:linfdBGN}, reducing each time step to two linear solves. Although exact geometric conservation is no longer guaranteed, the energy stability property in \eqref{eq:energy_dissipation} is still retained.

\subsubsection{Variant II: Robust startup strategy for rough initial data}
\label{sec:mixed_strategy}

The proposed scheme, consisting of \eqref{eqn:fd} and \eqref{eqn:fdBGN}, involves the time derivative of the mean curvature, $\partial_t^\circ \varkappa$, which imposes a regularity requirement on the initial surface $\Gamma^0$. If the initial surface lacks sufficient smoothness, the scheme may produce a large discrepancy between the evolution curvature $\varkappa^m$ and the geometric curvature $\kappa^m$.

In the case of nonsmooth initial data, we circumvent this issue by using the following startup strategy. During the first few time steps ($m < n_{\mathrm{switch}}$), we omit the curvature evolution equation. Instead, we find $\left(\mathscr{V}^{m+1}, \kappa^{m+1}, \vec X^{m+1} \right) \in \bV^h(\Gamma^m) \times \bV^h(\Gamma^m) \times [\bV^h(\Gamma^m)]^d$, together with the Lagrange multipliers $(\lambda^{m+1},\mu^{m+1})\in \bR^2$, such that
\begin{subequations}\label{eq:3d_startup_scheme}
\begin{align}
    &\ipd{\mathscr{V}^{m+1}, \varphi^h}_{{\Gm}} - \ipd{\nabs\kappa^{m+1}, \nabs\varphi^h}_{\Gm} + \ipd{\mathcal{W}^m\,[\kappa^{m+1}-\bkap], \varphi^h}_{\Gm} \nn\\
    &\qquad - \frac{1}{2}\ipd{[\kappa^m-\bkap]\,\kappa^m\,[\kappa^{m+1}-\bkap], \varphi^h}_{\Gm} \nn\\
    &\qquad - \lambda^{m+1}\ipd{1, \varphi^{h}}_{\Gamma^{m}} - \mu^{m+1}\ipd{\kappa^{m}, \varphi^{h}}_{\Gamma^{m}} = 0, \label{eq:3d_startup_scheme1}\\[0.5em]
    &\ipd{\frac{\vec{X}^{m+1} - \vec{X}^m}{\ttau} \cdot \vec{\nu}^{m+\frac{1}{2}}, \xi^h }^h_{\Gamma^m} - \ipd{\mathscr{V}^{m+1}, \xi^h}_{\Gamma^m} = 0, \label{eq:3d_startup_scheme2} \\[0.5em]
    &\ipd{\kappa^{m+1} \vec{\nu}^m, \vec{\eta}^h }^h_{\Gamma^m} + \ipd{\nabs \vec{X}^{m+1}, \nabs \vec{\eta}^h }_{\Gamma^m} = 0, \label{eq:3d_startup_scheme3} \\[0.5em]
    &\ipd{\mathscr{V}^{m+1}, 1}_{\Gamma^{m}} = 0, \label{eq:3d_startup_scheme4}\\[0.5em]
    &|\Gamma^{m+1}| - |\Gamma^0| = 0, \label{eq:3d_startup_scheme5}
\end{align} 
\end{subequations}
for all test functions $(\varphi^h, \xi^h, \vec{\eta}^h) \in \bV^h(\Gamma^m) \times \bV^h(\Gamma^m) \times [\bV^h(\Gamma^m)]^d$. Although unconditional energy stability no longer holds for this variant, it preserves the discrete enclosed volume and surface area exactly. This also leads to a nonlinear coupled system, which can be solved using the Newton iteration in a similar manner to \eqref{eq:newton_system}.

In practice, we monitor the relative energy dissipation rate 
\[\mathcal{R}_{\mathrm{diss}} := \frac{|E^m - E^{m-1}|}{\ttau\, E^0},\]
where $E^m = \frac{1}{2}\int_{\Gamma^m}(\kappa^{m})^2\dH^{d-1}$.
Once $\mathcal{R}_{\mathrm{diss}}$ drops below a user-defined threshold, e.g. $10^{-2}$, the surface is considered sufficiently regular and the algorithm switches back to the unconditionally stable and structure-preserving scheme.

\section{Numerical results}\label{sec:num}

In this section, we assess the accuracy, energy stability, and structure-preserving properties of the proposed fully discrete scheme through a series of experiments for 2D curves and 3D surfaces. The algorithm is implemented using the open-source finite element package NGSolve \cite{schoberl2014c++}. The resulting sparse linear systems are solved using the UMFPACK direct solver \cite{Davis04}. 

In all experiments, we first construct an initial polyhedral surface $\Gamma_Y$, where $\vec Y=\vec\id|_{\Gamma_Y}\in [\bV^h(\Gamma_Y)]^d$ denotes the identity map. To start the computation, we need $\varkappa^0$ and $\kappa^0$. If the initial surface is a sphere of radius $r_0$, we set $\varkappa^0 = \kappa^0 = -\frac{d-1}{r_0}$. Otherwise, we compute the two discrete mean curvatures using the BGN method with zero normal velocity. Namely, we find $(\delta\vec Y^0, \kappa^0)\in [\bV^h(\Gamma_Y)]^d\times\bV^h(\Gamma_Y)$ such that
\begin{subequations}\label{eqn:ic}
	\begin{align}
		&\ipd{\delta \vec Y^0\cdot\vec\nu_Y, \xi^h}_{\Gamma_Y}^h =0\qquad\forall\xi^h\in \bV^h(\Gamma_Y),\\
		&\ipd{\kappa^0\,\vec\nu_Y,~\vec\eta^h}_{\Gamma_Y}^h + \ipd{\nabs(\vec\id+\delta\vec Y^0),~\nabs\vec\eta^h}_{\Gamma_Y}=0\qquad\forall\vec\eta^h\in [\bV^h(\Gamma_Y)]^d,
	\end{align}
	\end{subequations}
	where $\vec\nu_Y$ is the unit normal of $\Gamma_Y$, defined analogously to \eqref{eq:vG}. We then set $\vec X^0 = \vec\id|_{\Gamma_Y} + \delta\vec Y^0$, $\Gamma^0 = \vec X^0(\Gamma_Y)$, and $\varkappa^0=\kappa^0$.	

Throughout the experiments, we monitor relative constraint violations and mesh regularity by defining the following discrete quantities:
\begin{equation*}
\Delta\vol^m = \frac{\vol(\Gm)-\vol(\Gamma^0)}{\vol(\Gamma^0)},\qquad
\Delta A^m = \frac{|\Gm|-|\Gamma^0|}{|\Gamma^0|},\qquad 
\Psi^m  =\frac{\max_{\sigma\in\mT^m}|\sigma|}{\min_{\sigma\in\mT^m}|\sigma|}.
\end{equation*}

\subsection{2D curve evolutions}

We note that for a simple closed planar curve with length preservation, we have 
\begin{equation*}
\int_{\Gamma}\varkappa\,\dH^1 = 2\pi,\qquad |\Gamma(t)|= A_0.
\end{equation*}
Then the energy in \eqref{eq:Henergy} can be recast as
\begin{align}
E_{\bkap}(\Gamma,\varkappa) = \frac{1}{2}\int_{\Gamma(t)}(\varkappa^2 - 2\,\bkap\,\varkappa + \bkap^2)\dH^1 = \frac{1}{2}\int_{\Gamma(t)}\varkappa^2\dH^1 - 2\pi\bkap + \frac{1}{2}\bkap^2 A_0,
\end{align}
which differs from the standard bending energy $E_0$ only by a constant. Thus, without loss of generality, we only consider the case of $\bkap = 0$ for all 2D examples.

\vspace{0.5em}
\noindent{\bf Example 1: Convergence test and geometric properties.}

\begin{table}[t]
\centering
\caption{Errors and convergence, as well as discrete geometric quantities in the evolution of a $2\times 1$ ellipse, where $h_0= \frac{1}{64}$ and $\ttau=\mathcal{O}(h^2)$.}
\label{tab:exp1_convergence}
\begin{tabular}{ccccc}
\hline
$J$ & $e_{h,\ttau}(T=1)$ & order & $\max_{0\leq m\leq M}|\Delta A^m|$ & $\max_{0\leq m\leq M}|\Delta\vol^m|$  \\
\hline
64   & --          & --     & 9.16E-16  & 1.43E-13   \\
128  & 4.40E-3 & --     & 1.28E-15  & 1.10E-13   \\
256  & 1.10E-3 & 1.98 & 1.83E-15  & 2.63E-13   \\
512  & 3.31E-4 & 1.87 & 3.48E-15  & 1.12E-12  \\
1024 & 8.38E-5 & 1.86 & 4.58E-15  & 4.37E-12   \\
2048 & 2.20E-5 & 1.93 & 8.07E-15  & 1.74E-11   \\
\hline
\end{tabular}
\end{table}

We start with a convergence experiment by considering the relaxation of a $2\times 1$ ellipse up to $T=1.0$. Since no exact solution is available, we assess the numerical errors by comparing solutions on successively refined meshes using the manifold distance $\mathrm{M}(\Gamma_h, \Gamma_{h/2})$ \cite{Zhao2021energy}, where the time step is chosen as $\Delta t = \mathcal{O}(h^2)$ with $h=\frac{1}{J}$. Precisely, we measure the errors between two closed planar curves $\Gamma_1$ and $\Gamma_2$ by the symmetric difference of their enclosed domains, i.e.,
\begin{equation}
\mathrm{M}(\Gamma_1, \Gamma_2) := | \Omega_1 \triangle \Omega_2 | = |\Omega_1| + |\Omega_2| - 2|\Omega_1 \cap \Omega_2|,
\end{equation}
where $\Omega_1$ and $\Omega_2$ denote the regions enclosed by $\Gamma_1$ and $\Gamma_2$, respectively. Then the errors are computed as
\begin{equation*}
e_{h,\ttau}(t) = \mathrm{M}(\Gamma_{h,\ttau}(t), \Gamma_{2\,h, 4\ttau}(t)),
\end{equation*}
where $\Gamma_{h,\ttau}(t) = \vec X_{h,\ttau}(\Gamma^m, t)$ for $t\in [t_m, t_{m+1}]$ and $\vec X_{h,\ttau}(\cdot, t)$ is defined via
\begin{equation*}
		\vec X_{h,\ttau} (\vec q, t) = \frac{t_{m+1}-t}{\ttau}\vec q + \frac{t-t_m}{\ttau}\vec X^{m+1}(\vec q),\quad\forall\vec q\in \mQ^m, \quad t\in[t_m, t_{m+1}].
	\end{equation*}

The numerical results are reported in Table~\ref{tab:exp1_convergence}, which demonstrates the optimal second-order convergence rate. Meanwhile, the maximum relative losses in length and enclosed area remain at the level of machine precision throughout the evolution. The slight increase in the relative enclosed-area loss is likely due to solver tolerances.

\vspace{0.5em}
\noindent{\bf Example 2: Evolution of an 8:1 elongated tube.} 

\begin{figure}[!htp]
    \centering
    \includegraphics[width=0.85\linewidth]{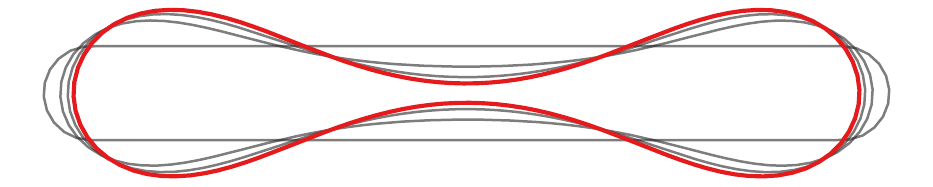}
    \caption{Evolution of an 8:1 elongated tube toward a dumbbell steady state.}
    \label{fig:ex2_evolution}
\end{figure}

\begin{figure}[!htp]
    \centering
    \includegraphics[width=0.95\linewidth]{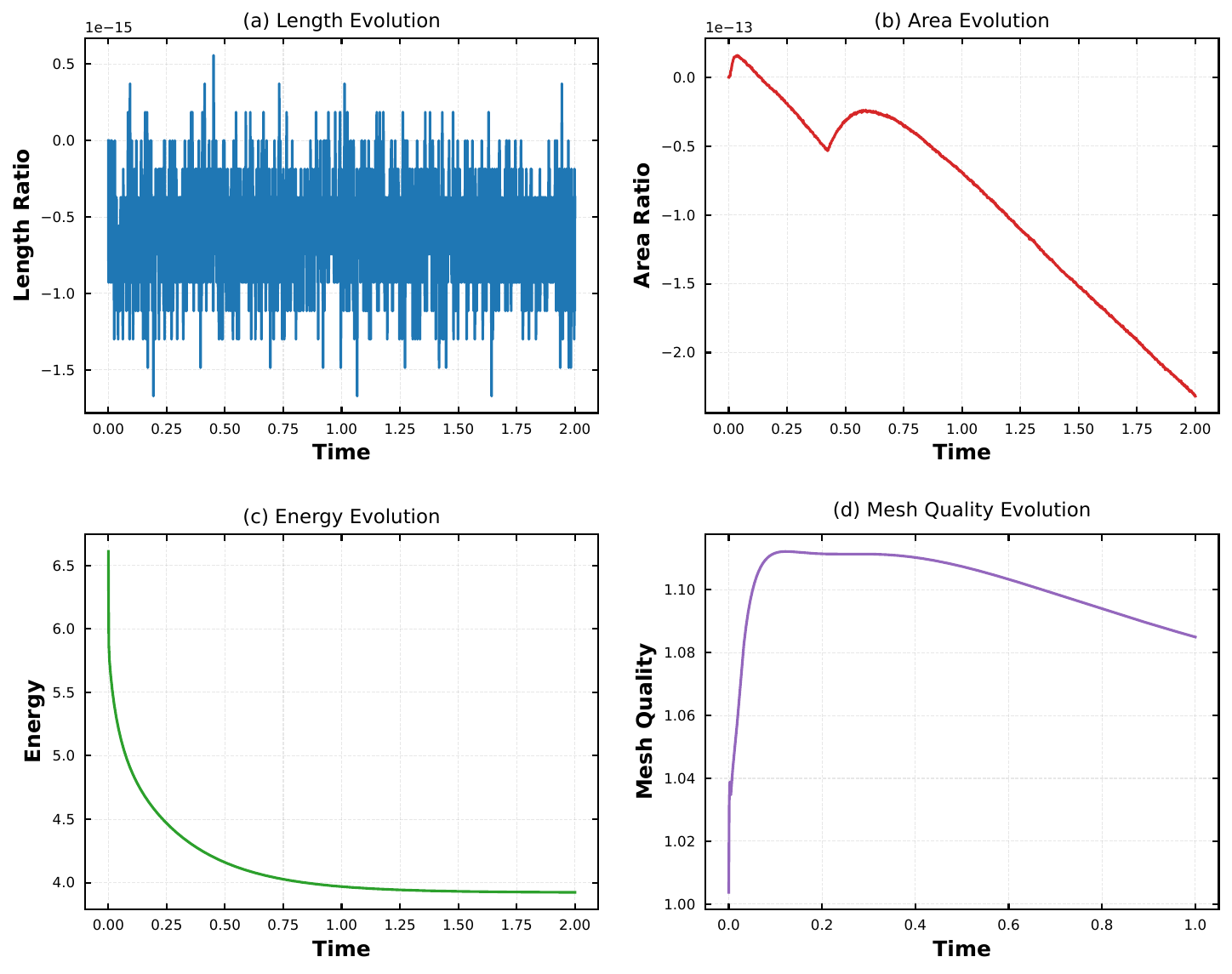}
    \caption{Time plots of the discrete quantities in the evolution of the $8\times 1$ elongated tube.}
    \label{fig:ex2_metrics}
\end{figure}

To test the performance of our method under large deformations, we consider a capsule-like tube with an 8:1 aspect ratio as the initial curve. We use $J=128$ and $\ttau = 10^{-3}$, and visualize the results in Figure~\ref{fig:ex2_evolution}. Here, the curve evolves toward a symmetric dumbbell-shaped steady state corresponding to a minimal-energy configuration.

We also plot the time histories of the discrete quantities in Figure~\ref{fig:ex2_metrics}. We observe the exact preservation of the discrete enclosed area and length, together with a monotonic energy decay. Furthermore, the mesh ratio $\Psi^m$ increases slightly at the initial stage and then gradually decreases, indicating that the implicitly generated BGN tangential motion effectively prevents vertex clustering. Therefore, artificial remeshing is generally not required.

\vspace{0.5em}
\noindent{\bf Example 3: Startup strategy for nonsmooth initial data.} 

We next evaluate the startup strategy, i.e., the {\bf Variant II} scheme in Section~\ref{sec:mixed_strategy}. For the initial curve, we consider the boundary of a $2 \times 2$ square with a $0.8 \times 1$ rectangular notch. This curve is $C^0$-continuous but has sharp corners. For the discretization parameters, we use $J=256$, $\ttau = 10^{-3}$, and $T=3.0$. As visualized in Figure~\ref{fig:ex4_evolution}, the {\bf Variant II} scheme acts as a robust startup until $t=10^{-3}$ to handle the nonsmooth initial data. Then we switch to the main scheme \eqref{eqn:fd}--\eqref{eqn:fdBGN}. In general, we observe preservation of the area and length to machine precision throughout.

\begin{figure}[!htbp]
    \centering
    \includegraphics[width=0.85\linewidth]{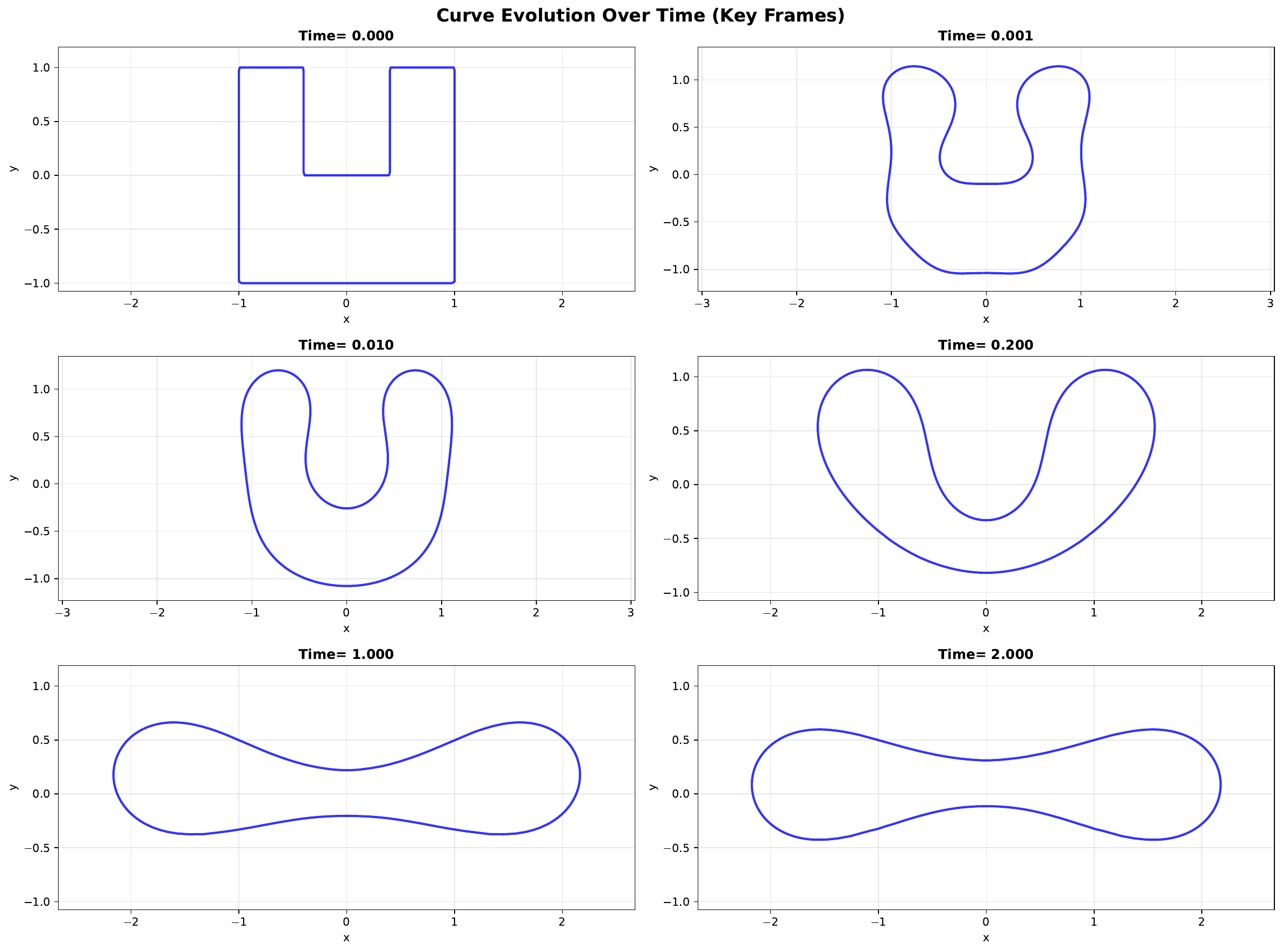}
    \caption{Snapshots in the evolution of the nonsmooth initial curve, which is taken as the boundary of a notched square.}
    \label{fig:ex4_evolution}
\end{figure}

\begin{figure}[!htp]
    \centering
    \includegraphics[width=0.85\linewidth]{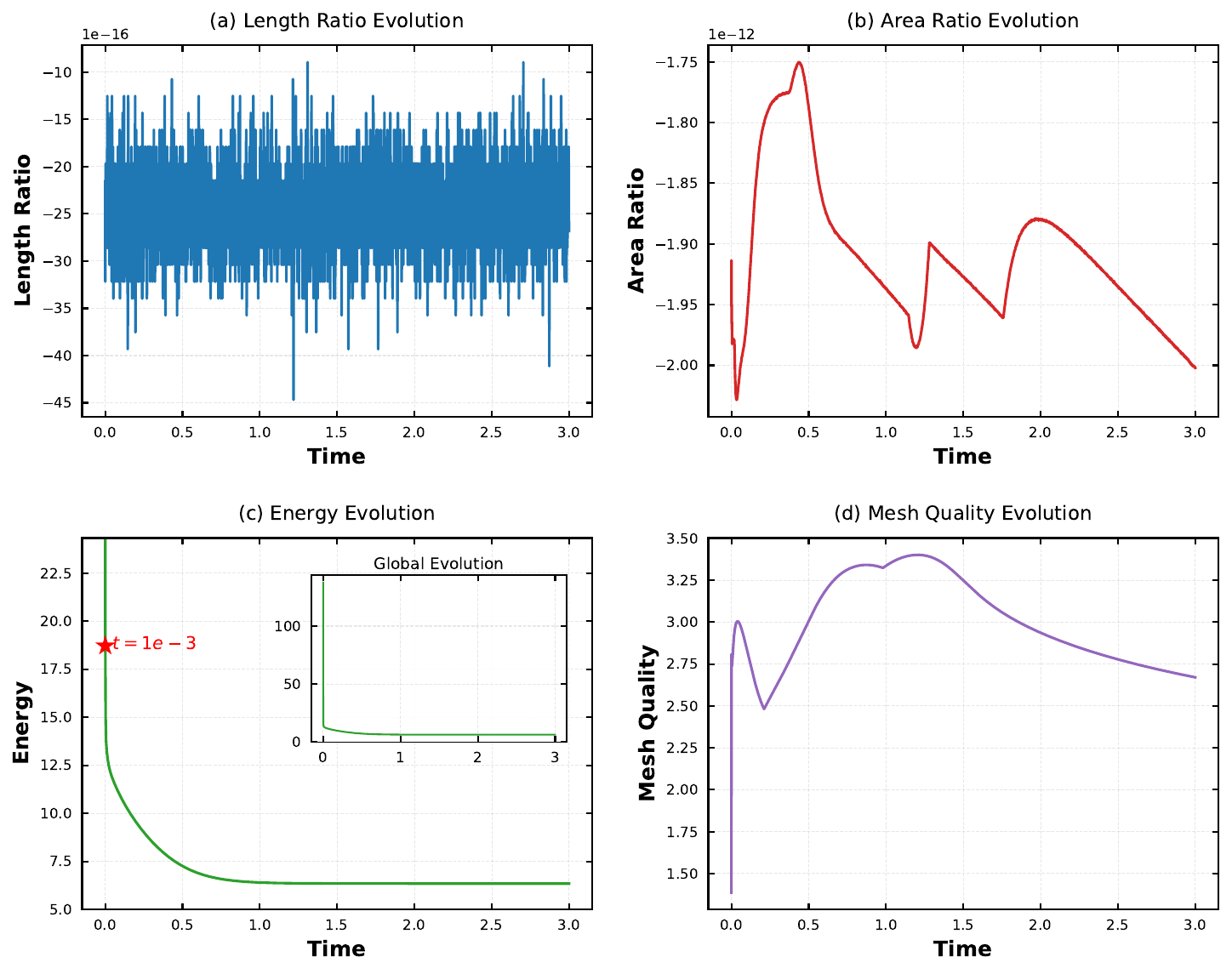}
    \caption{Time histories of the discrete quantities in the evolution of the nonsmooth initial curve, which is taken as the boundary of a notched square.}
    \label{fig:ex4_metrics}
\end{figure}

\subsection{3D surface dynamics}

The morphology of lipid bilayer membranes governed by the Helfrich energy is strongly influenced by the spontaneous curvature $\bkap$ and the dimensionless reduced volume, which is defined as follows (see \cite{SBK1991}):
\[0 < v := \frac{3V}{4\pi (A/4\pi)^{3/2}} \le 1.\]
For the symmetric initial geometries considered here, the constrained gradient flow typically evolves toward the classical \textit{prolate-dumbbell} or \textit{oblate-discocyte} metastable branches~\cite{MJK2023helfrich,BONITO2010}. 

\vspace{0.5em}
\noindent{\bf Example 4: The oblate-discocyte branch.}

We first consider an initial $4\times4\times 1$ oblate ellipsoid ($v \approx 0.586$) with 7,730 vertices, using $\bkap = 0$ and $\Delta t = 10^{-3}$. Figure~\ref{fig:3d1} shows the surface successfully relaxing into a classical biconcave discocyte, mimicking the morphology of human red blood cells. The quantitative results in Figure~\ref{fig:3d_ep1__metrics} also confirm exact preservation of the volume and surface area up to machine precision. The decay of the discrete energy is observed as well.  

\begin{figure}[!htp]
\centering
\includegraphics[width=0.4\linewidth]{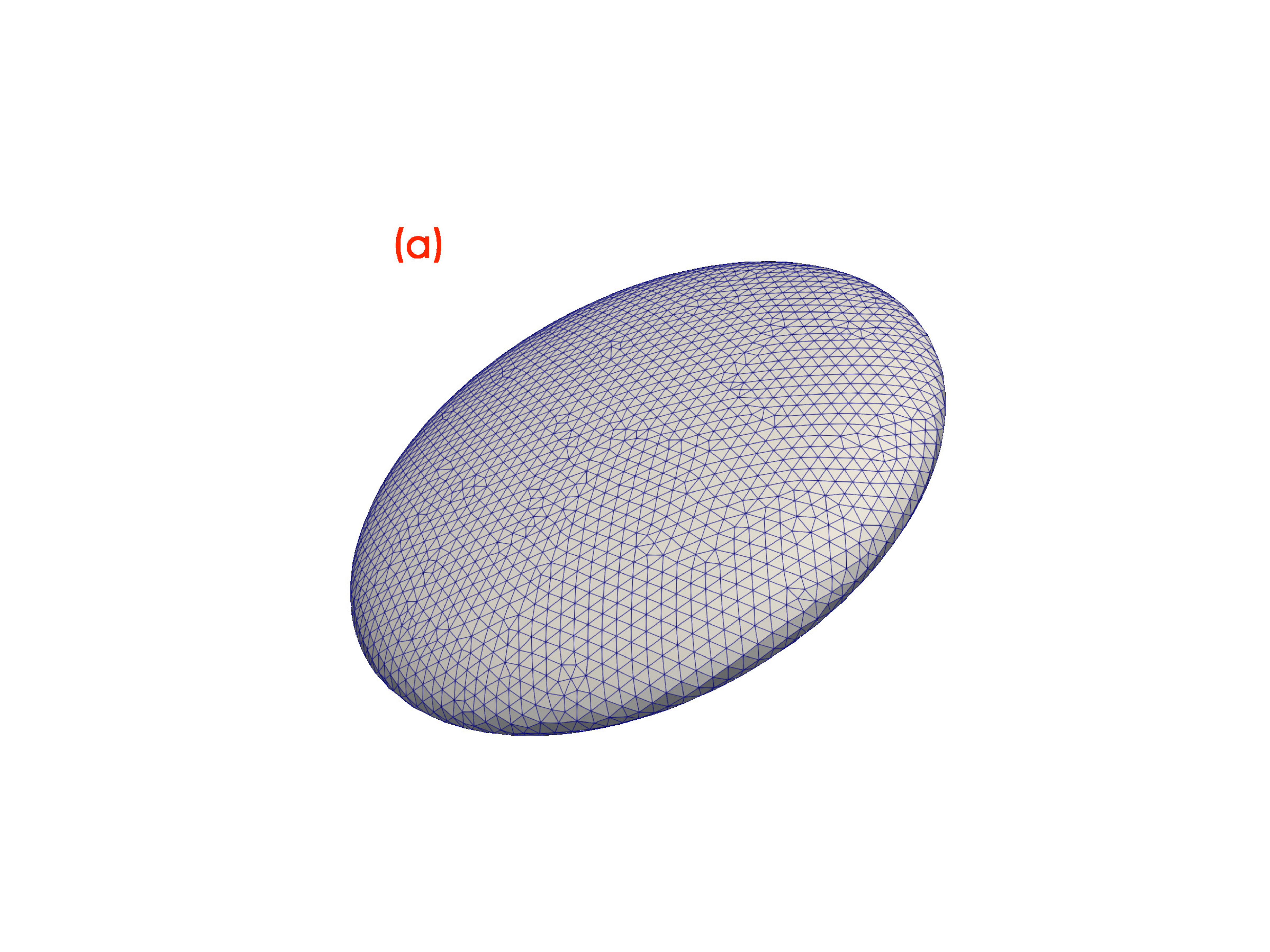}
\includegraphics[width=0.4\linewidth]{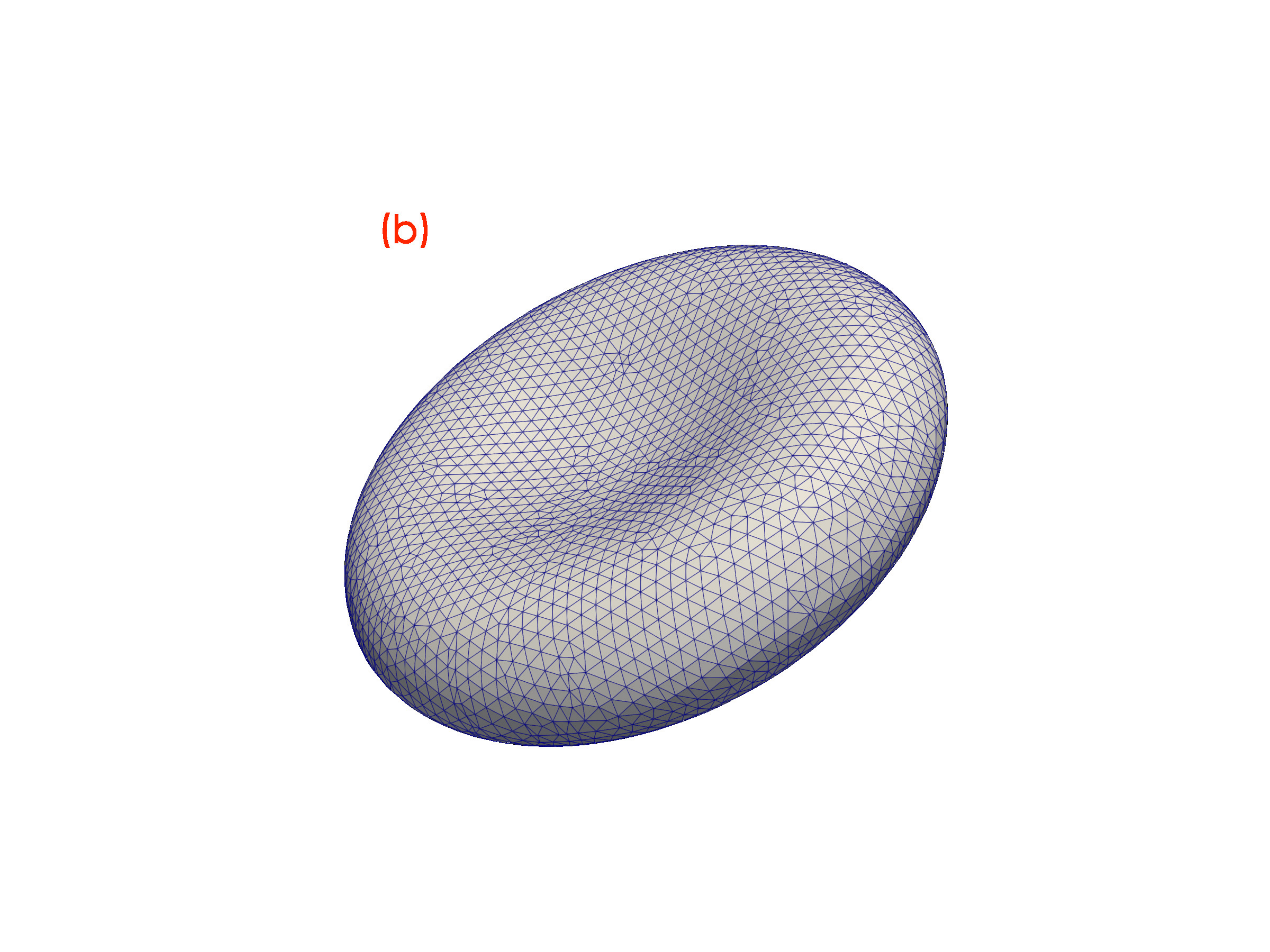}\\
\includegraphics[width=0.4\linewidth]{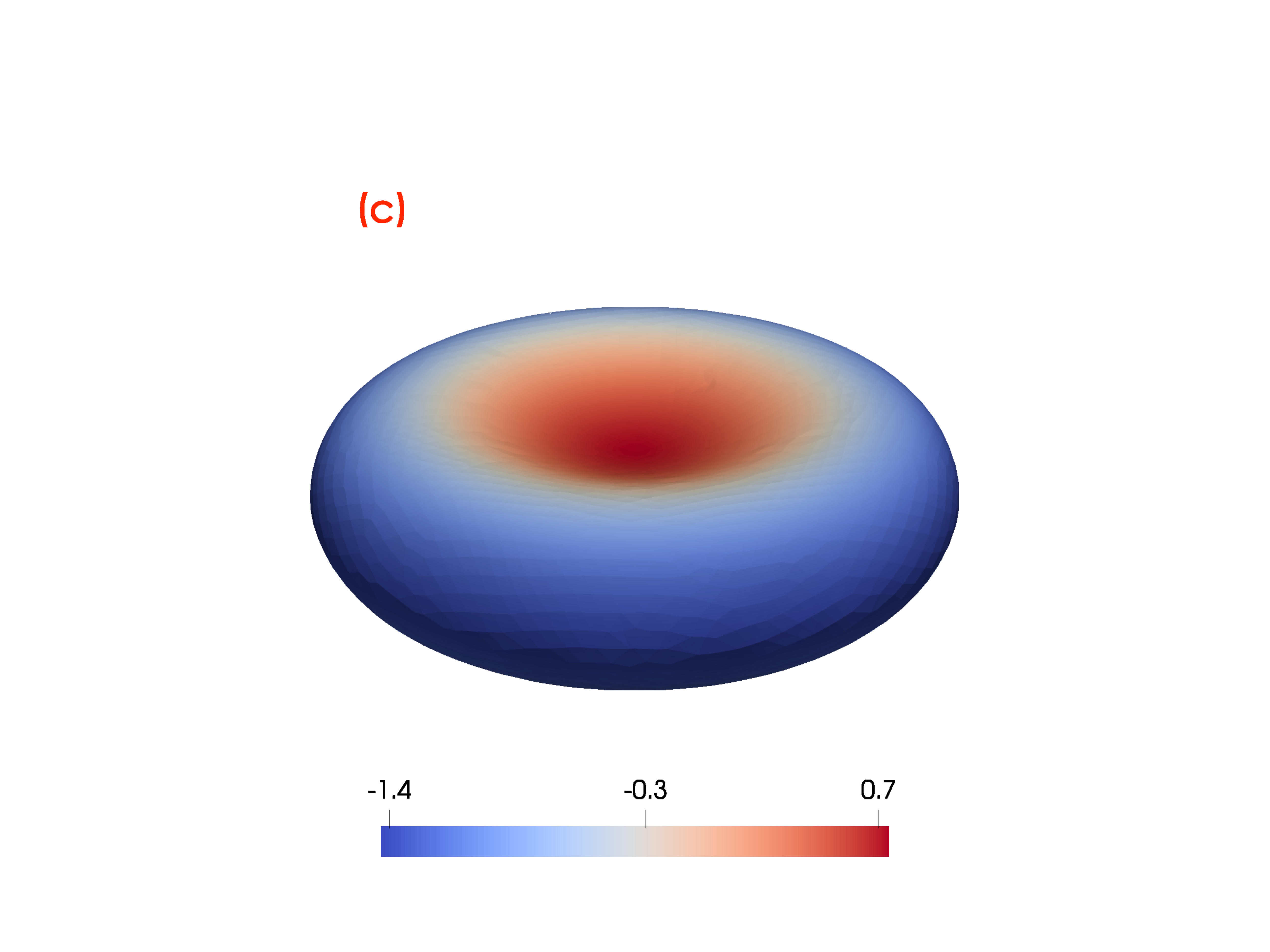}
\includegraphics[width=0.4\linewidth]{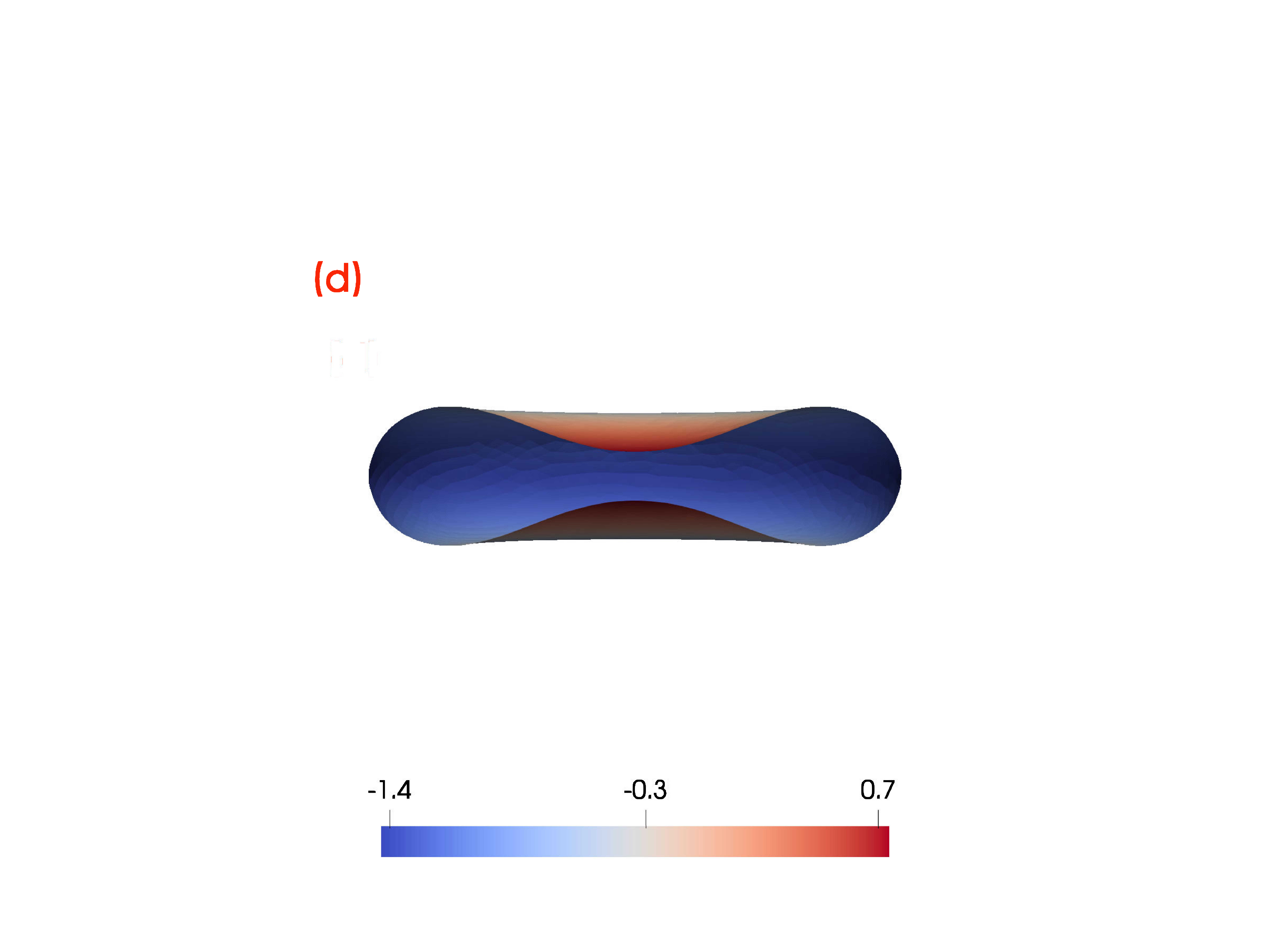}
\caption{[$\bkap=0$] Relaxation of a $4\times 4\times 1$ oblate ellipsoid into a discocyte shape. (a) $t=0$; (b) $t=0.7$; (c) curvature at $T=1$; (d) the cross-section profile at $T=1$.}
\label{fig:3d1}
\end{figure}

\begin{figure}[!htbp]
    \centering
    \includegraphics[width=0.8\linewidth]{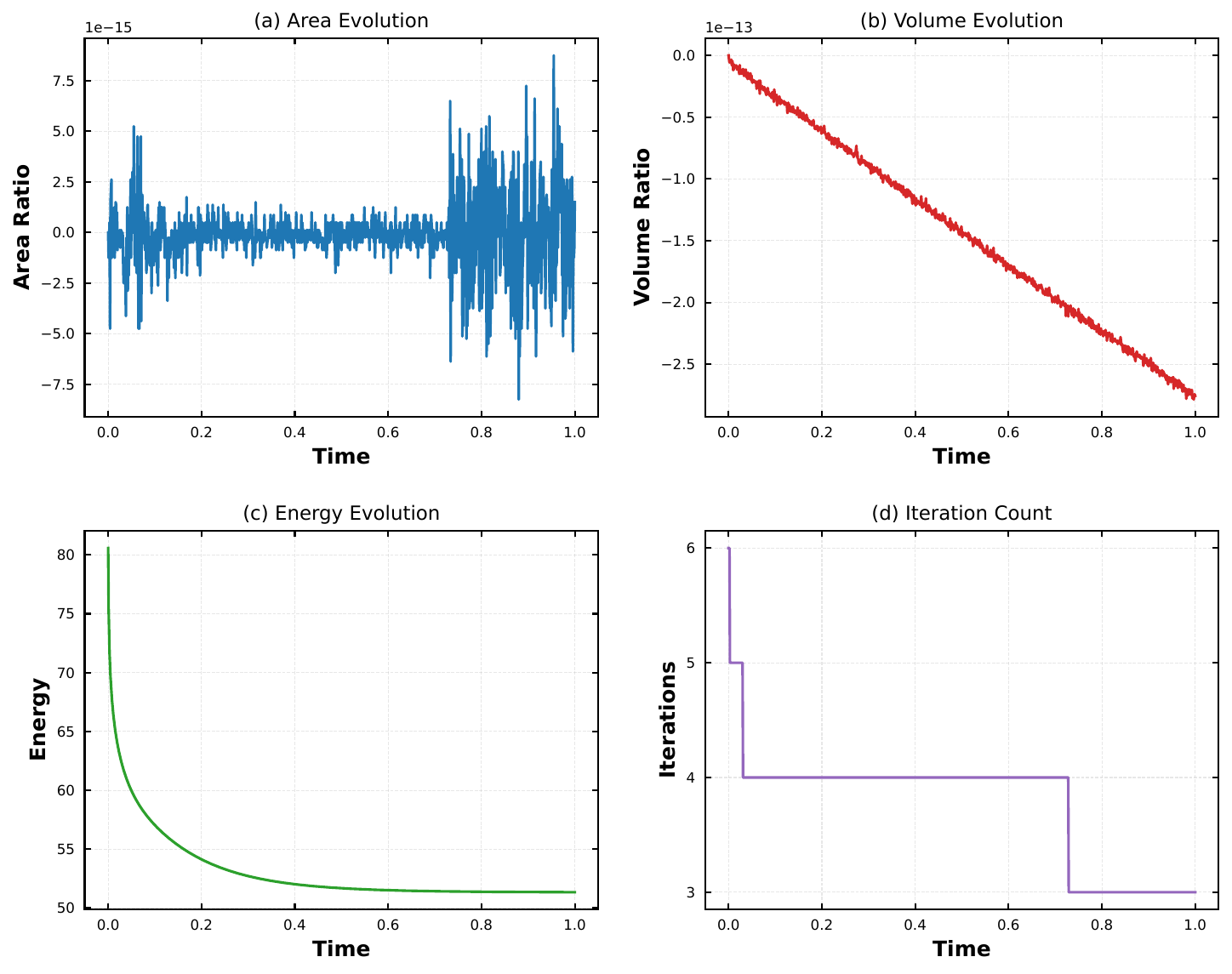}
    \caption{[$\bkap=0$] Time histories of the discrete quantities in the evolution of the 3D oblate ellipsoid.}
    \label{fig:3d_ep1__metrics}
\end{figure}

\vspace{0.5em}
\noindent{\bf Example 5: The prolate branch and spontaneous curvature effects.} 

We systematically vary the reduced volume by setting the initial shapes to $n\times 1\times 1$ prolate ellipsoids ($n \in \{2, 4, 6, 8\}$). For $\bkap = 0$ (Figure~\ref{fig:3d_ep2_bkap_0}), the energy minimization triggers the formation of a pronounced neck, driving the surfaces into the \textit{prolate-dumbbell} branch. When a nonzero spontaneous curvature $\bkap = -1.2$ is introduced (Figure~\ref{fig:3d_ep2_bkap_m1}), the preferred local mean curvature alters the equilibrium profiles, resulting in significantly thicker necks. Our scheme robustly captures these delicate, physically driven morphological bifurcations.

\begin{figure}[!htbp]
    \centering
    \subfloat[$2:1:1$]{\includegraphics[width=0.24\linewidth]{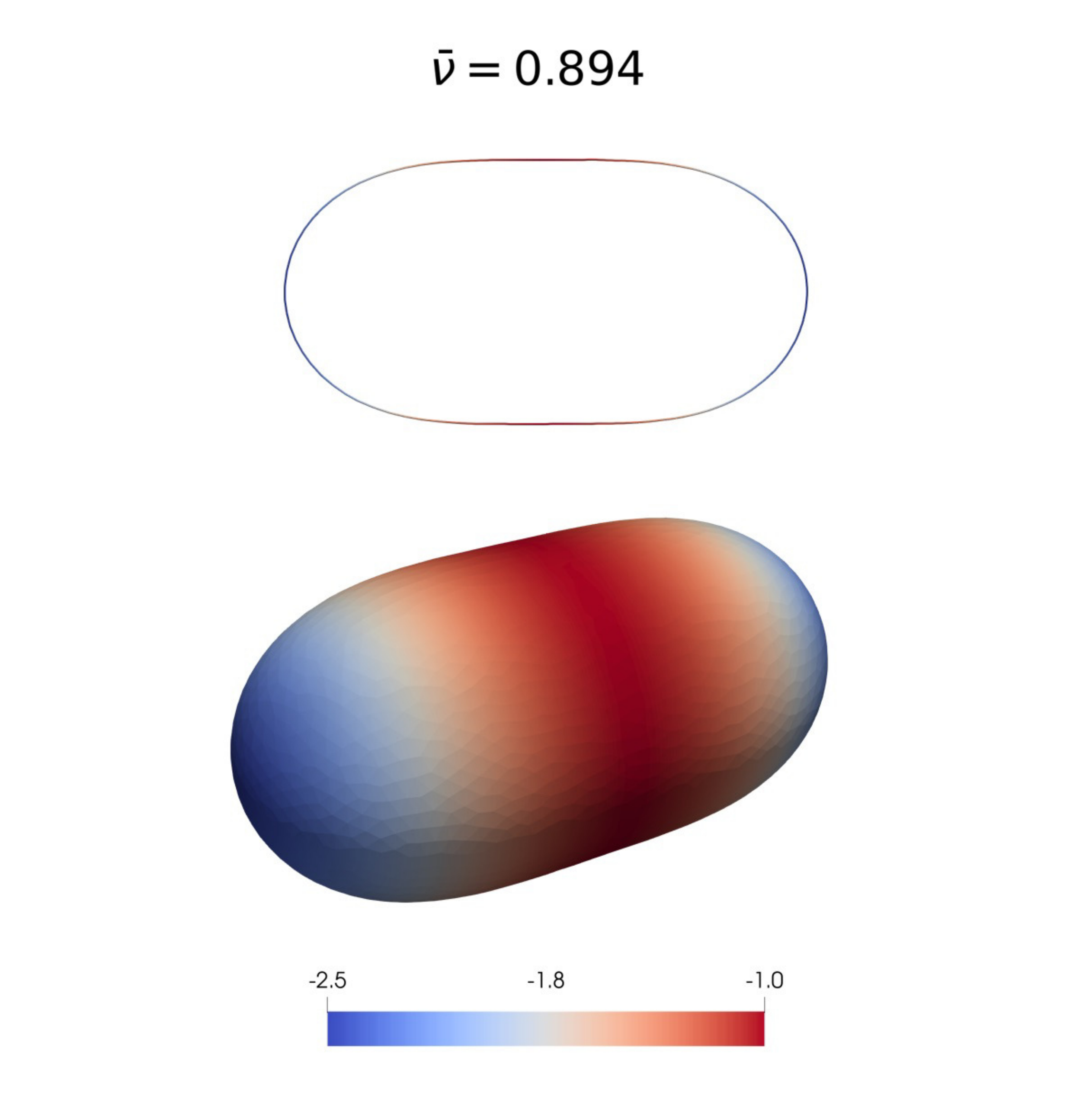}} \hfill
    \subfloat[$4:1:1$]{\includegraphics[width=0.24\linewidth]{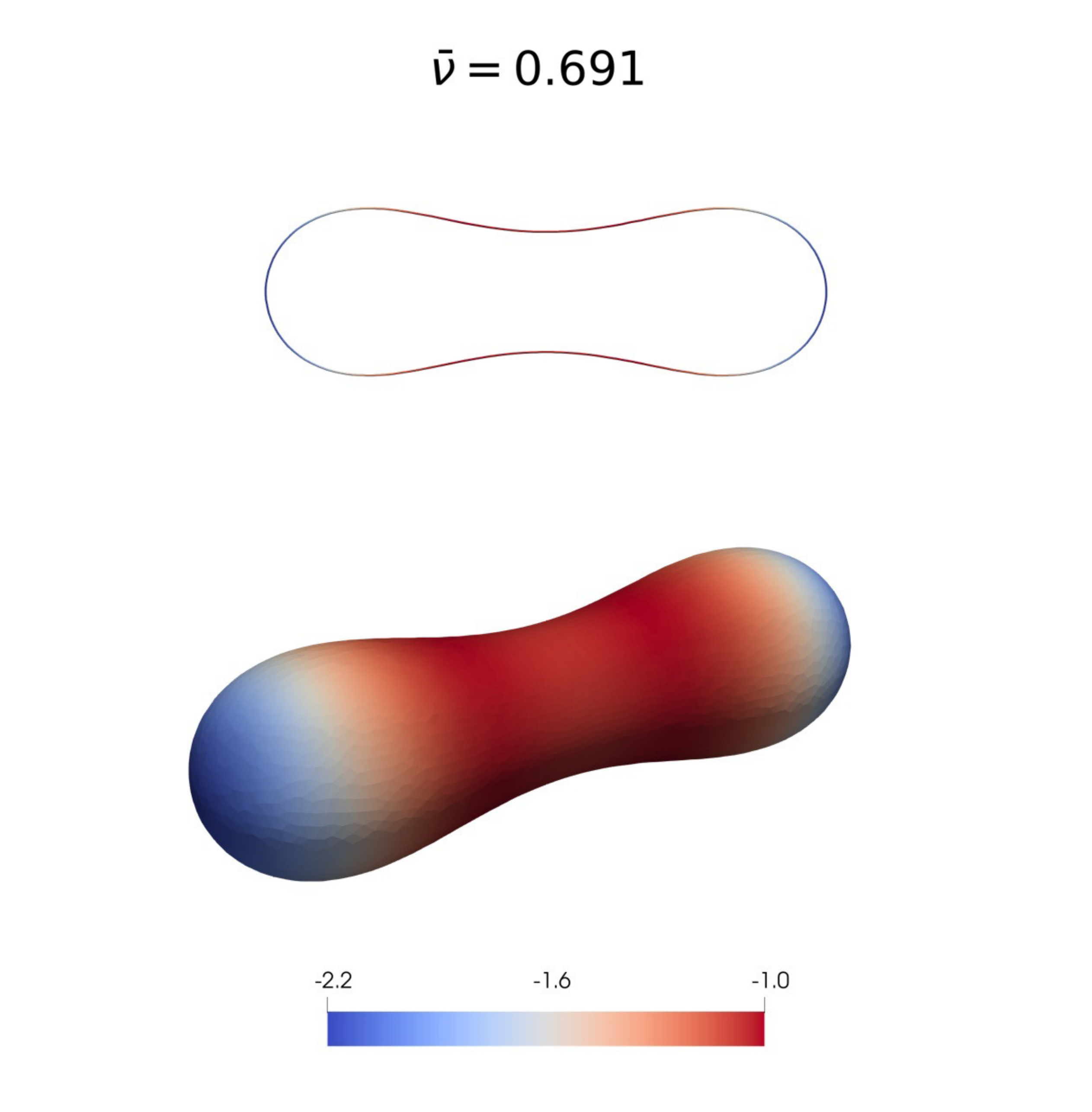}} \hfill
    \subfloat[$6:1:1$]{\includegraphics[width=0.24\linewidth]{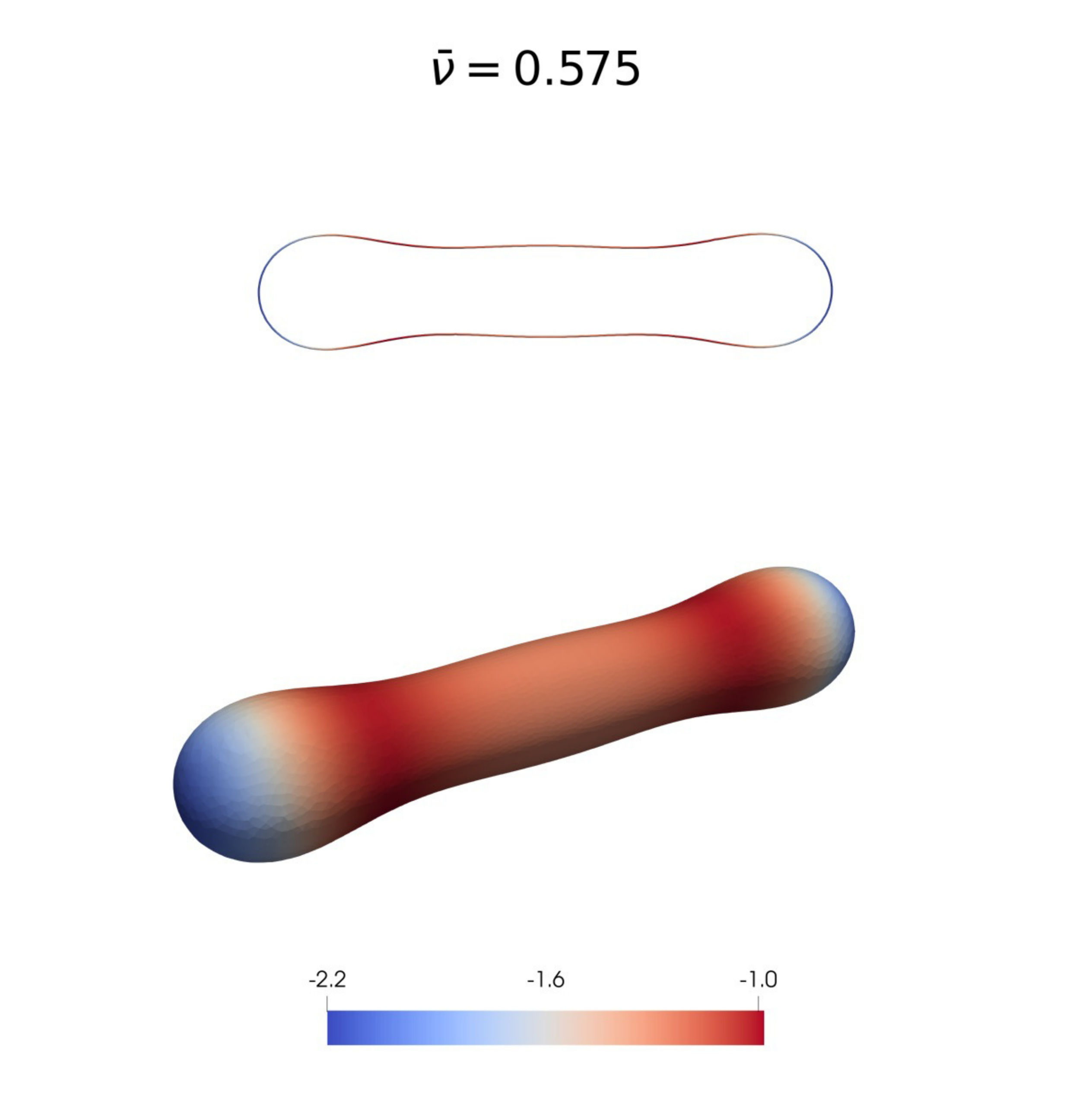}} \hfill
    \subfloat[$8:1:1$]{\includegraphics[width=0.24\linewidth]{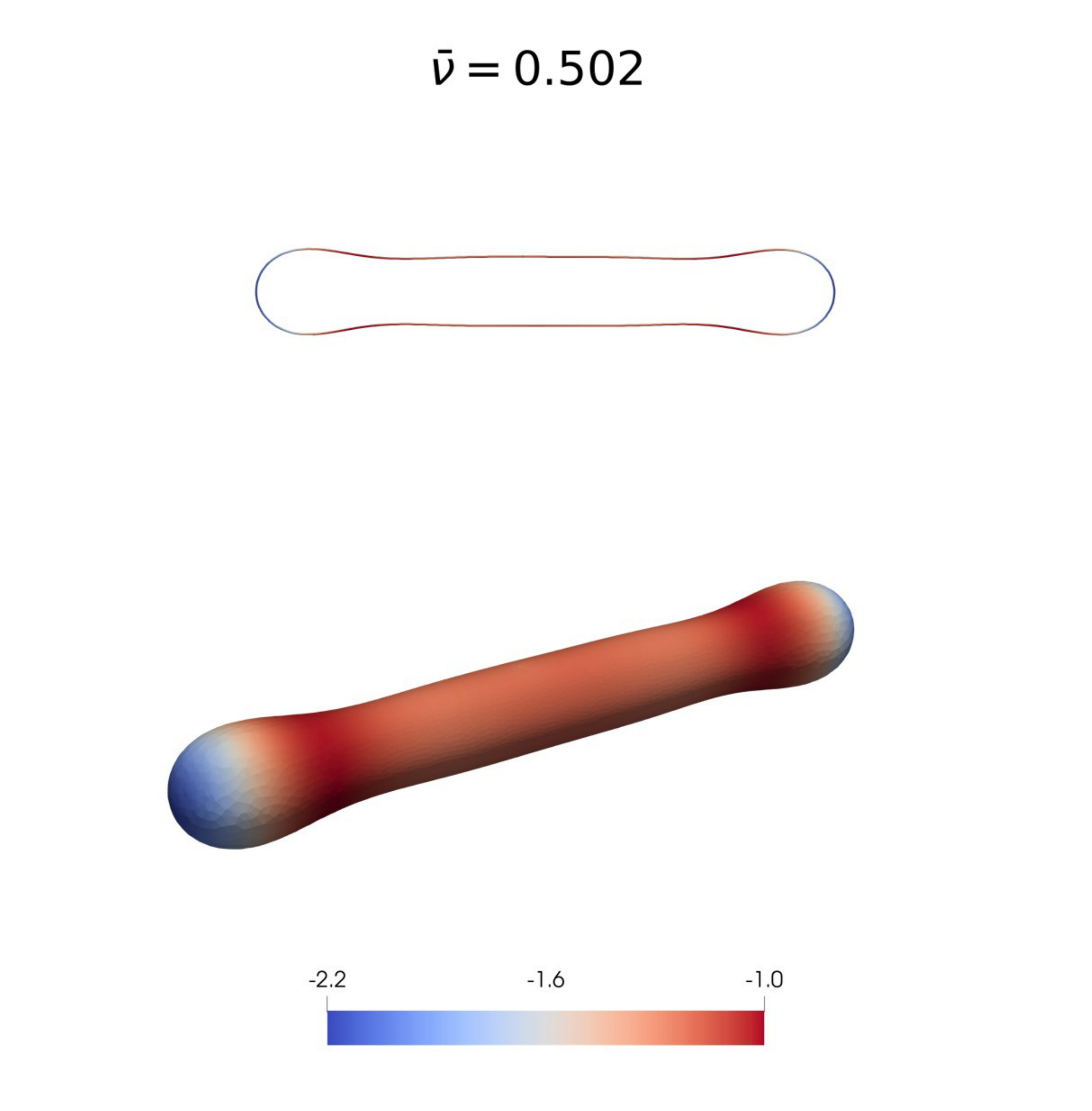}}  
    \caption{[$\bkap=0$] Visualizations of the surface morphologies at the final time $T=1.0$ for prolate ellipsoids of different dimensions.}
    \label{fig:3d_ep2_bkap_0} 
\end{figure}

\begin{figure}[!htbp]
    \centering
    \subfloat[$2:1:1$]{\includegraphics[width=0.24\linewidth]{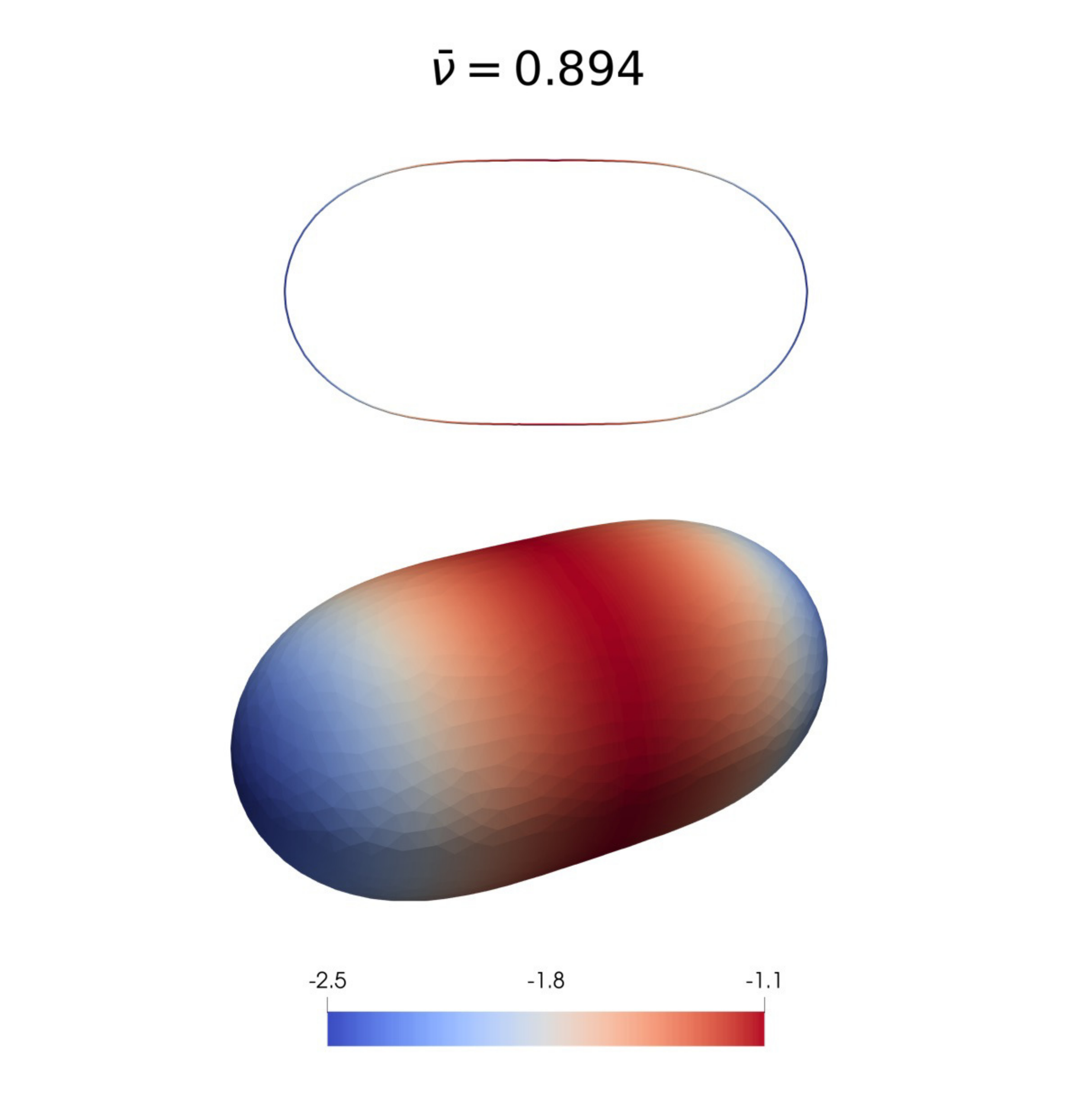}} \hfill
    \subfloat[$4:1:1$]{\includegraphics[width=0.24\linewidth]{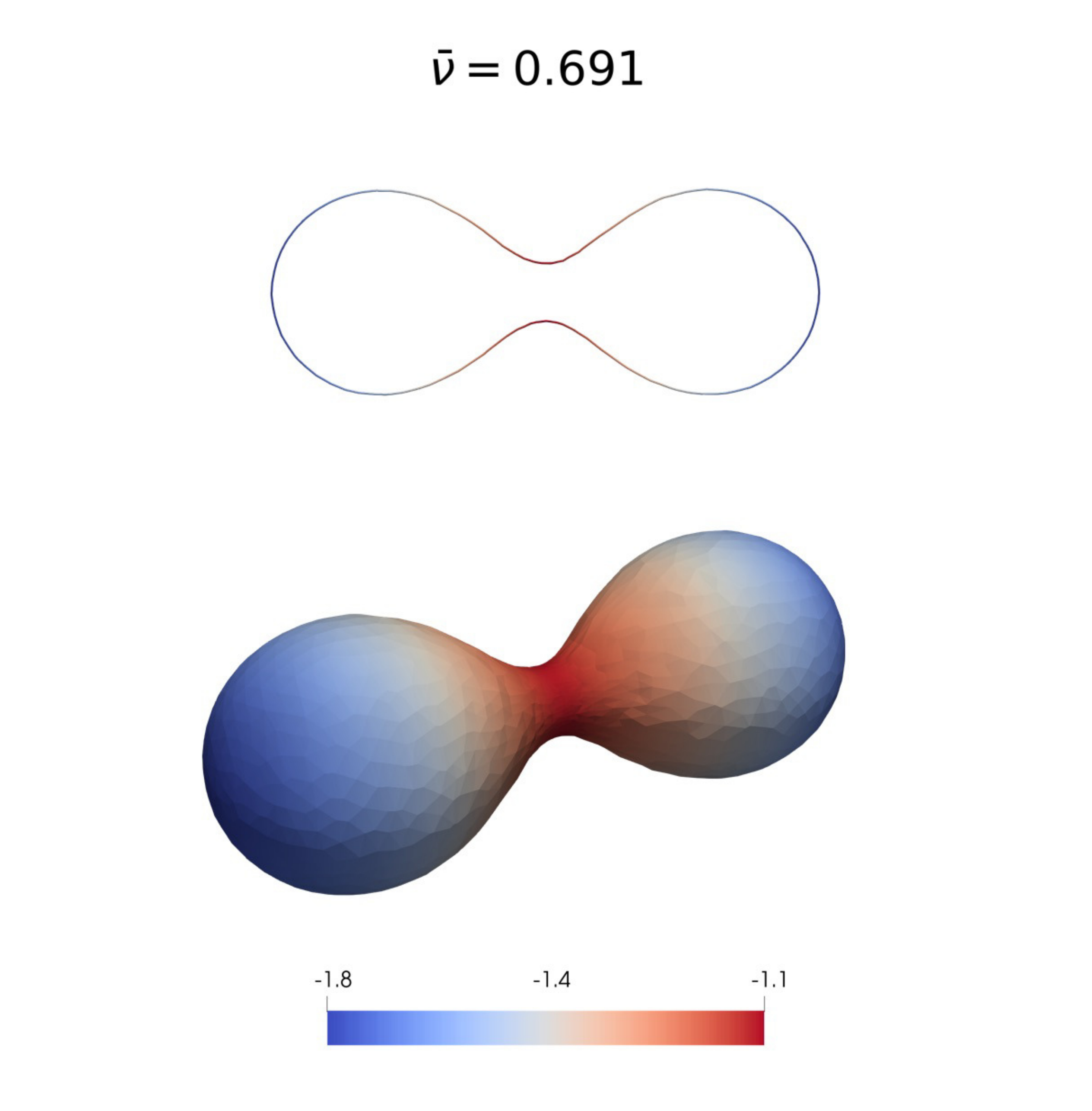}} \hfill
    \subfloat[$6:1:1$]{\includegraphics[width=0.24\linewidth]{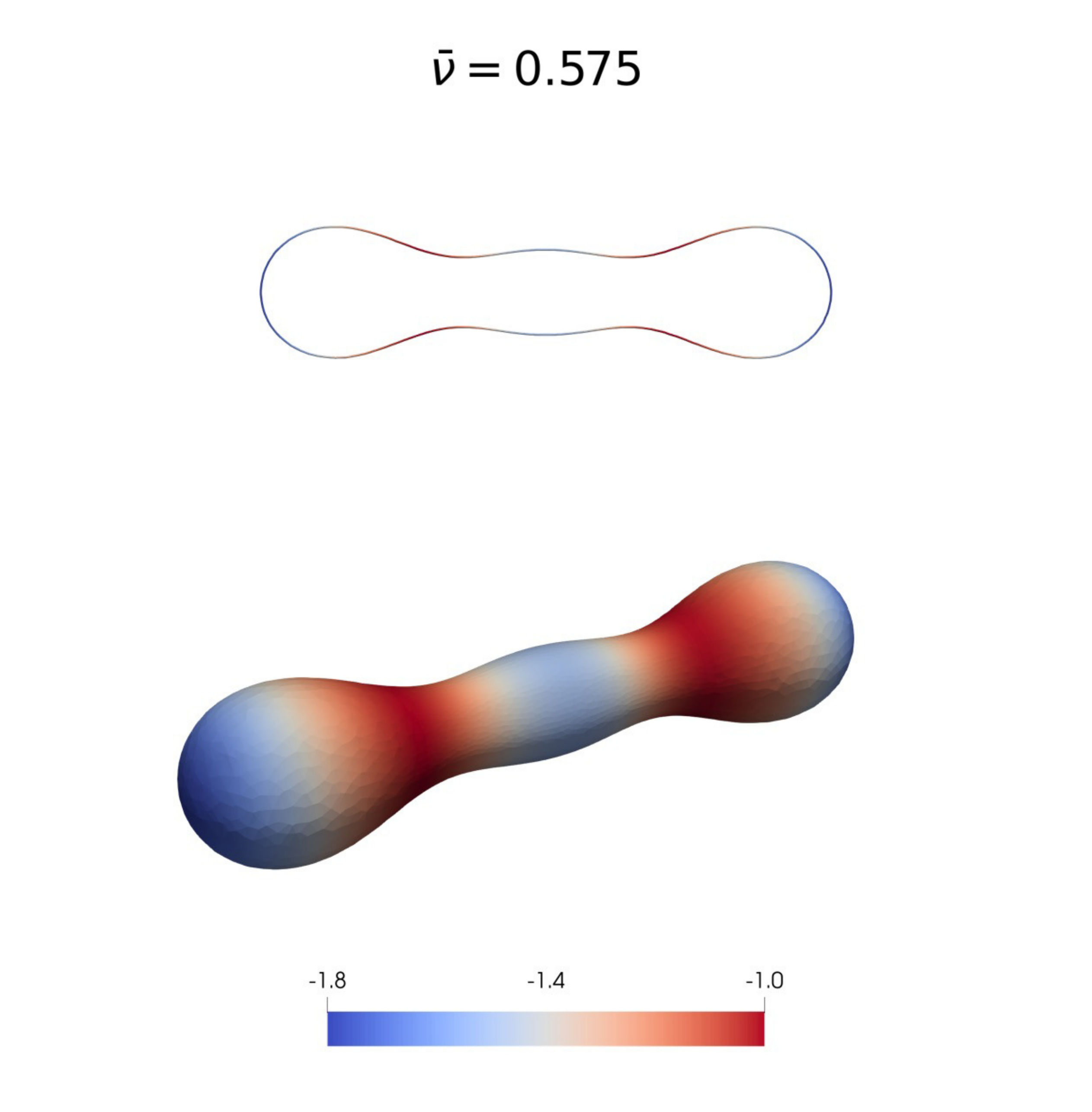}} \hfill
    \subfloat[$8:1:1$]{\includegraphics[width=0.24\linewidth]{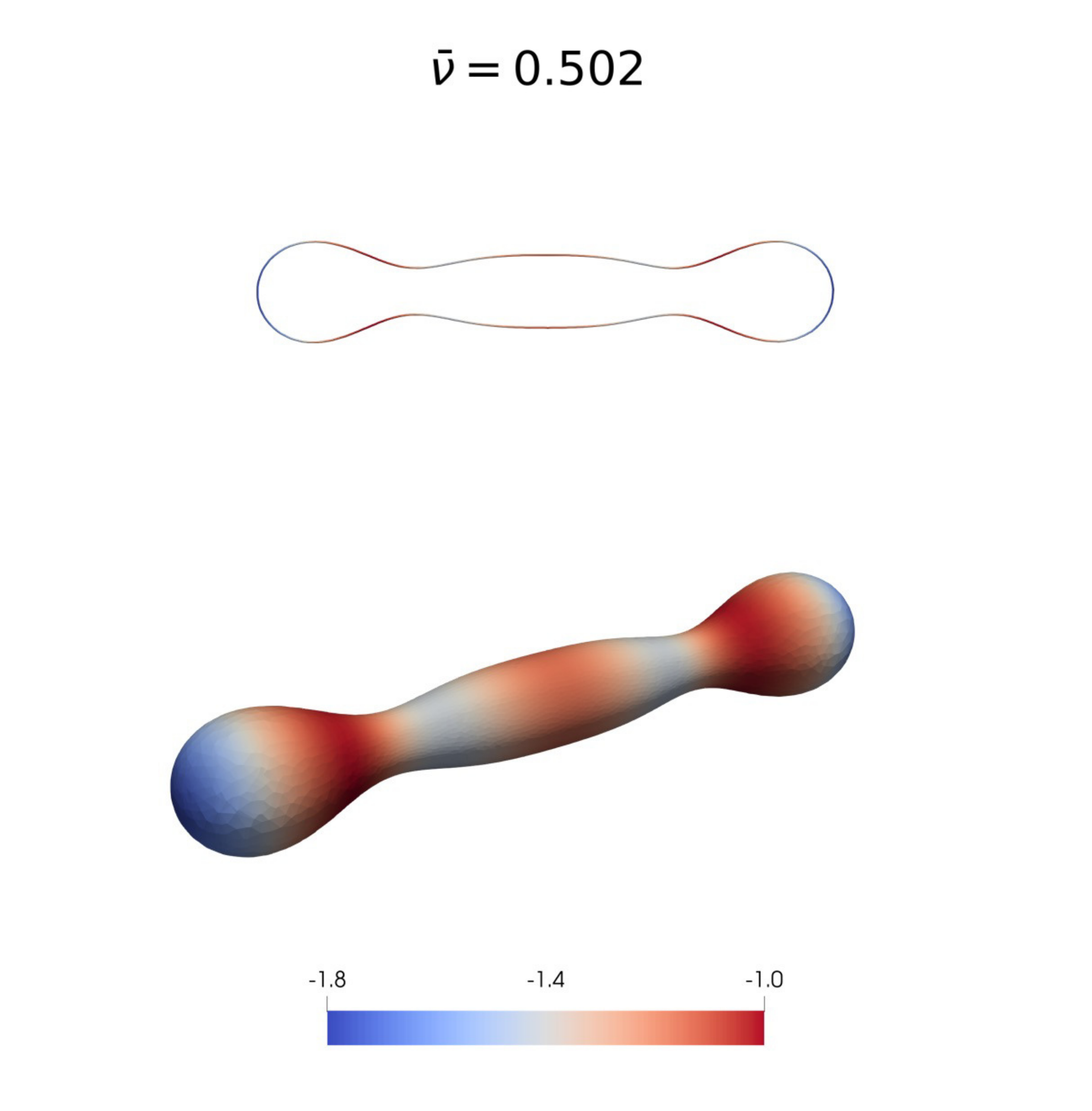}}
    \caption{[$\bkap=-1.2$] Visualizations of the surface morphologies at the final time $T=2$. The nonzero spontaneous curvature alters the neck radii.}
    \label{fig:3d_ep2_bkap_m1}
\end{figure}

\vspace{0.5em}
\noindent{\bf Example 6: Startup strategy for singular 3D data.} 

To verify the robustness of the startup strategy in 3D, we start with an initial standard cuboid of dimension $4\times 4\times 1$. The initial surface is only $C^0$-continuous with sharp corners, which poses some difficulty for our main scheme due to the singular initial curvature. We then deploy the startup scheme (Variant II) for just 10 steps with the time step size $\Delta t = 10^{-4}$, which rapidly smooths out the singular edges (see Figure~\ref{fig:3d_ep2}(b)). After this startup, we switch to the main scheme with the time step size $\Delta t = 10^{-3}$. As demonstrated in Figure~\ref{fig:3d_ep4_metrics}, this startup strategy is effective and preserves the volume and surface area throughout the evolution. Moreover, the BGN tangential motion maintains good mesh quality during the
  relaxation toward a smooth equilibrium shape.

\begin{figure}[!htbp]
    \centering
    \subfloat[$t=0$ (Cuboid)]{\includegraphics[width=0.24\linewidth]{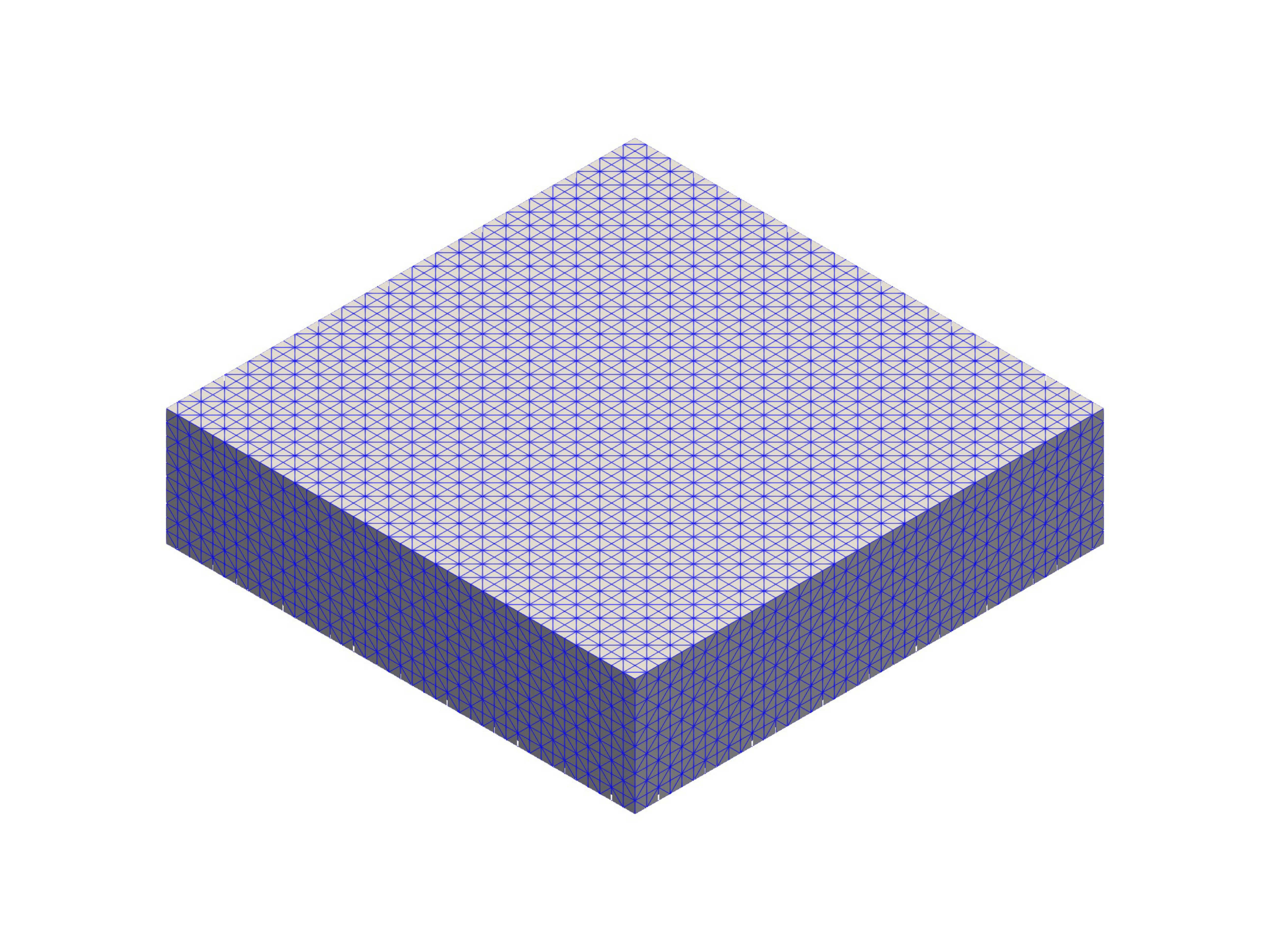}} \hfill
    \subfloat[$t=10^{-3}$]{\includegraphics[width=0.24\linewidth]{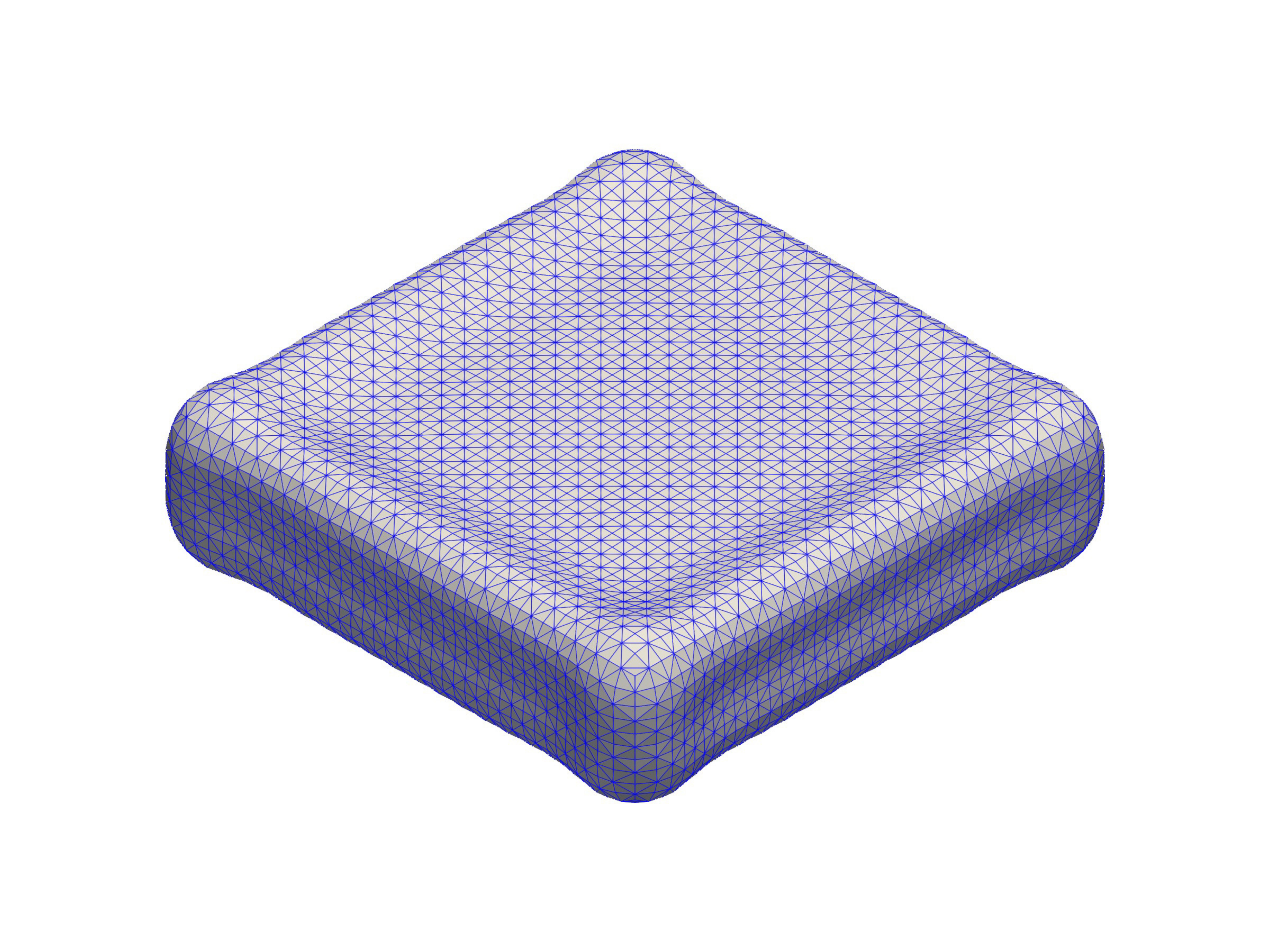}} \hfill
    \subfloat[$t=0.1$]{\includegraphics[width=0.24\linewidth]{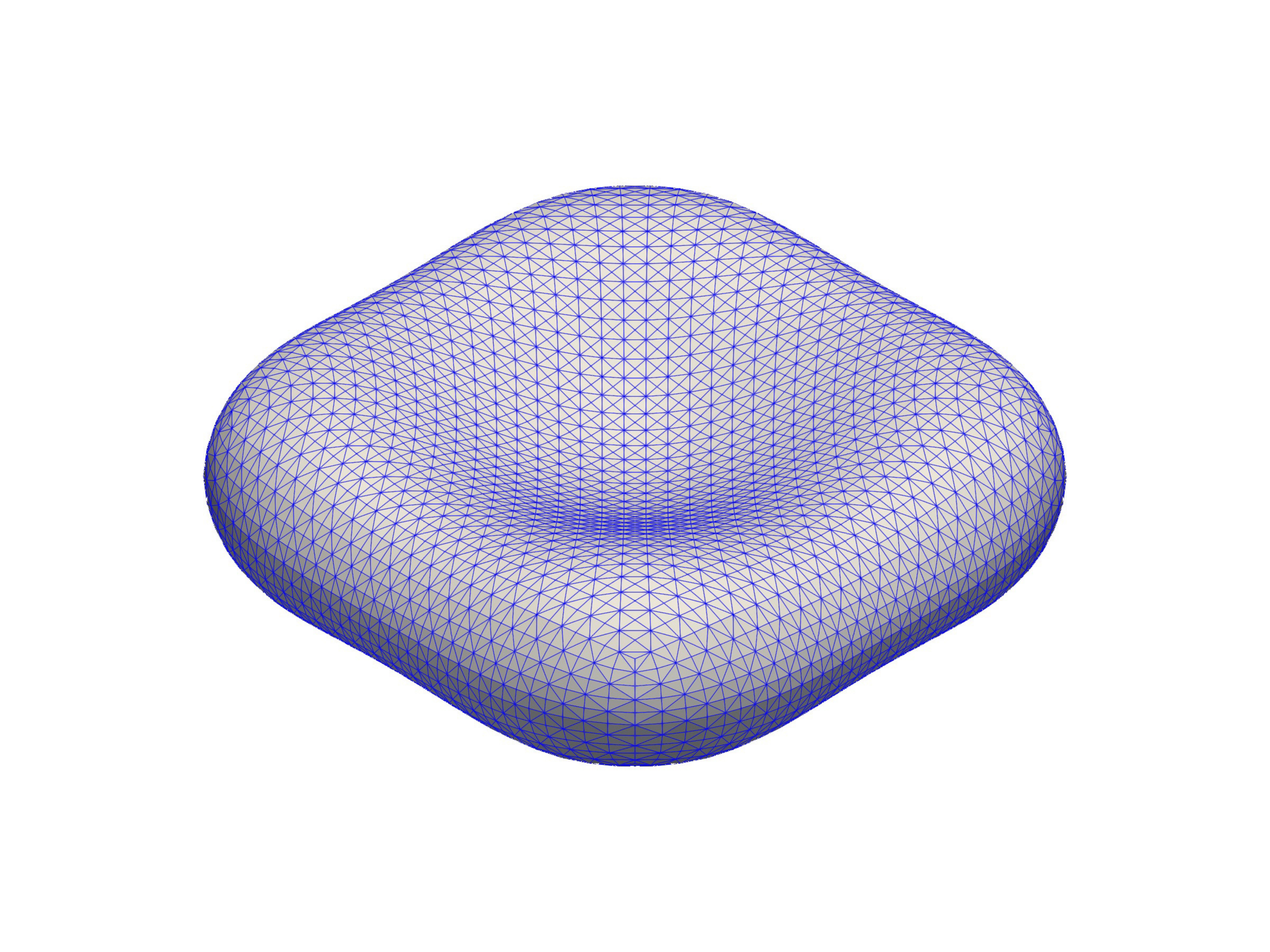}} \hfill
    \subfloat[$t=0.3$]{\includegraphics[width=0.24\linewidth]{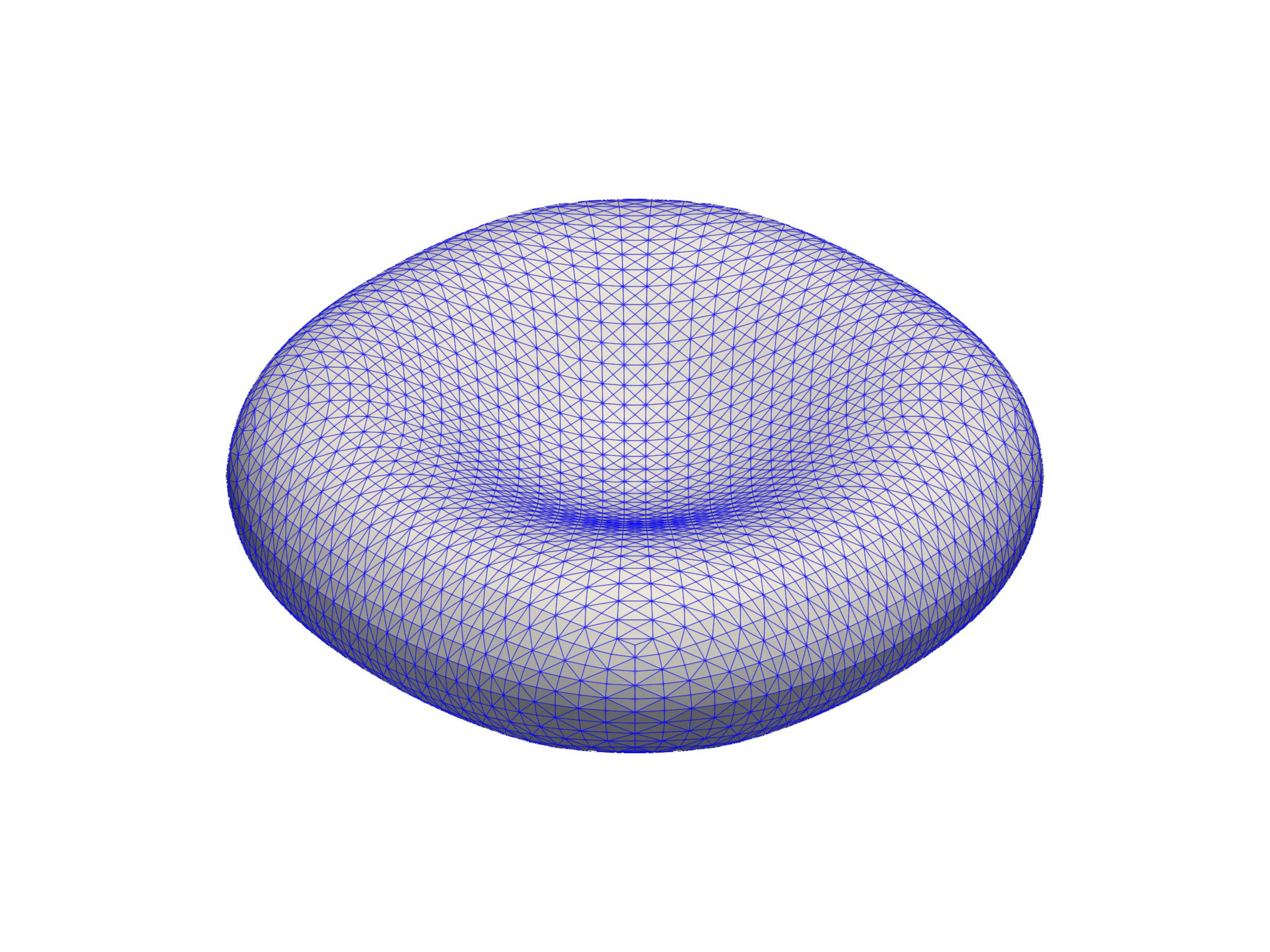}}
    \caption{[$\bkap=0$] Snapshots in the evolution of an initially nonsmooth cuboid, where we employ the {\bf Variant II} scheme for the startup strategy.}
    \label{fig:3d_ep2}
\end{figure}

\begin{figure}[!htp]
    \centering
    \includegraphics[width=0.85\linewidth]{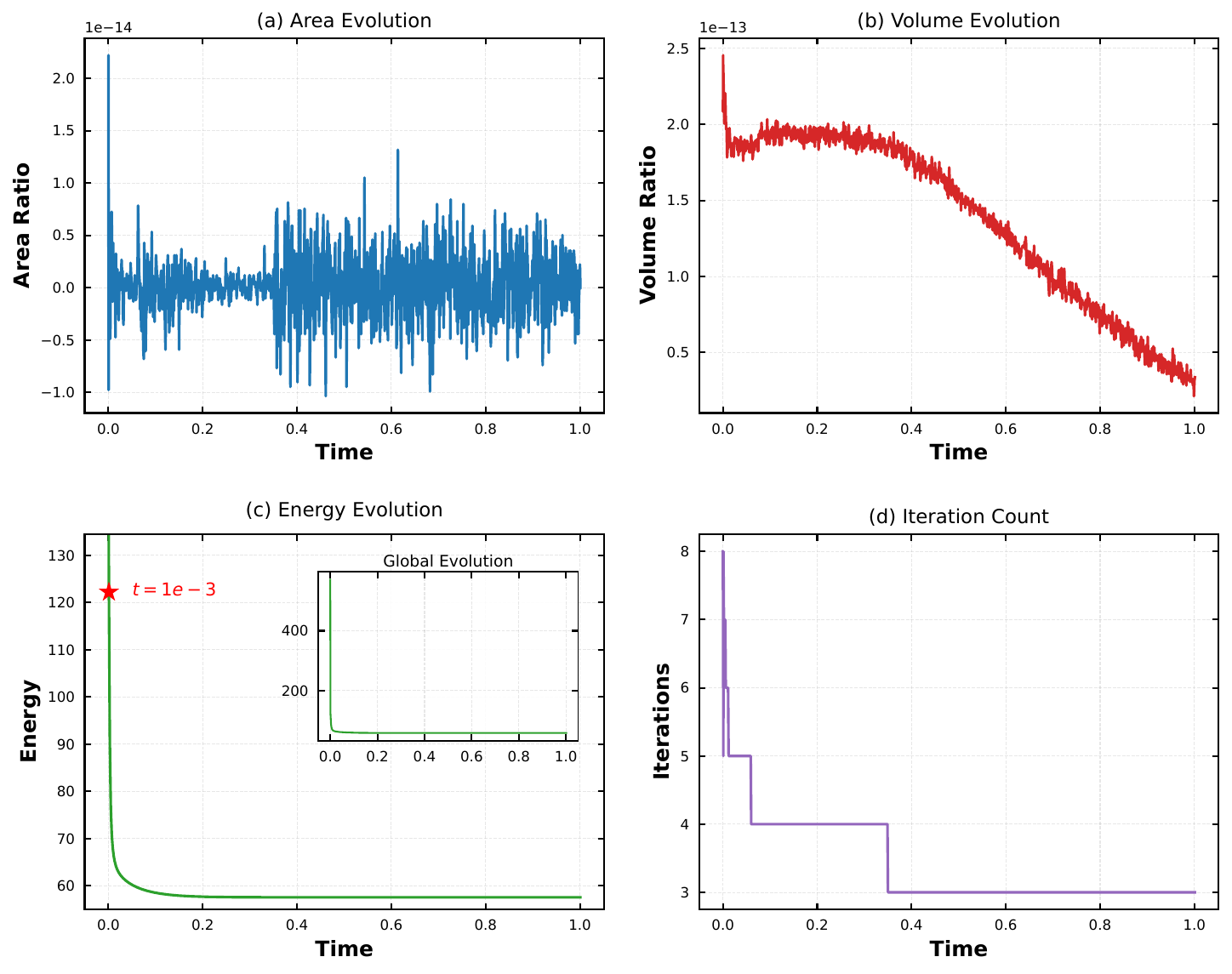}
    \caption{Time histories of discrete quantities in the evolution of an initially nonsmooth cuboid.}
    \label{fig:3d_ep4_metrics}
\end{figure}

\section{Conclusions}\label{sec:con}

We have proposed a geometric structure-preserving parametric finite element method for the constrained Helfrich flow of closed curves and surfaces. The method is built on a two-part splitting of the velocity: the normal velocity is first computed from the gradient-flow structure, and this velocity is then applied together with a suitable tangential velocity to move the mesh and correct the geometric constraints.

The first part of the method is the gradient-flow approximation. By using the curvature evolution equation and imposing the volume and surface area constraints through the normal velocity, this step yields a fully discrete linear system. The resulting approximation satisfies an unconditional energy dissipation estimate, while the normal-velocity constraints provide the discrete counterpart of the continuous volume and area conservation laws.

The second part applies the computed normal velocity together with the BGN tangential velocity. This step serves as the actual mesh update and, at the same time, corrects the enclosed volume and surface area as well as the computation of the geometric curvature. The time-weighted interface normal gives exact volume preservation, while the area-correction multiplier enforces exact surface area preservation. Thus the full scheme combines energy decay, exact preservation of the geometric constraints, and good mesh quality in a single framework.

We also discussed two variants of the method, which reflect different practical tradeoffs. The fully linear variant keeps the energy-stable gradient-flow step but gives up exact geometric correction, while the startup variant improves robustness for nonsmooth or singular initial data before switching to the main structure-preserving scheme. Possible extensions include high-order parametric finite element discretizations and generalizations to more complex membrane energies or additional physical constraints.

\section*{Acknowledgements}

This work was partially supported by the National Natural Science Foundation of China (No. 12401572, Q.Z.) and the Key Project of the National Natural Science Foundation of China (No. 12494555, Q.Z.).

\appendix

\section{Differential calculus}
Let $\Gamma(t)\subset\bR^d$ be an evolving hypersurface without boundary, with the velocity $\mathscr{\vv V}$ defined in \eqref{eq:velocity}. 
For a sufficiently smooth function $f$ defined on $\Gamma(t)$, we recall the Reynolds transport theorem (see, e.g., \cite[Theorem~3.2]{Barrett20}).
\begin{align}\label{eq:transport}
    \ddt\int_{\Gt} f \dH^{d-1} &= \int_{\Gt} \bigl( \partial_t^\circ f + f \nabs\cdot\mathscr{\vv V} \bigr) \dH^{d-1}\nn\\
    &= \int_{\Gamma(t)}\bigl(\partial_t^\square f - f\,\mathscr{V}\,\varkappa\bigr)\dH^{d-1},
\end{align}
where $\partial_t^\circ$ denotes the material time derivative  
\begin{equation*}
	\partial_t^\circ f = \ddt f(\vec x(\vec\rho,t), t)\quad\forall(\vec\rho,t)\in\Upsilon\times [0,T],
\end{equation*}
which follows the parameterization \eqref{eq:para}, and $\partial_t^\square$ stands for the normal time derivative:
\begin{equation} \label{eq:normalf}
	\partial_t^\square f = \partial_t^\circ f-\vec{\mathscr{V}}\cdot\nabs f\qquad\mbox{on}\quad\Gt,
\end{equation}
which measures the change of $f$ on the moving surface in the normal direction.

\bibliographystyle{abbrv}
\bibliography{bib}

\end{document}